\documentclass[twoside,11pt]{article}
\date{\today}

\title{The Mixing method: low-rank coordinate descent for semidefinite programming with diagonal constraints}

\usepackage{jmlr2e}
\usepackage{todo}

\usepackage[utf8]{inputenc} 
\usepackage[T1]{fontenc}    
\usepackage{url}            
\usepackage{booktabs}       
\usepackage{amsfonts}       
\usepackage{nicefrac}       
\usepackage{microtype}      

\usepackage{amsmath,cases,amssymb,verbatim, amsfonts}
\usepackage{algorithm, algpseudocode}%
\usepackage{algorithmicx}
\usepackage{graphicx, wrapfig}
\usepackage{sidecap}
\usepackage{xspace}
\usepackage{float,subcaption}

\newtheorem{assumption}[theorem]{Assumption}
\newtheorem{claim}[theorem]{Claim}

{\left ( \begin{smallmatrix}}%
{\end{smallmatrix} \right )}%
{\left \| \begin{smallmatrix}}%
{\end{smallmatrix} \right \|}%

\newcommand{\gr} {\nabla}

\newcommand{\norm}[1] {\|#1\|}
\newcommand{\mdot}[2]{\langle#1,#2\rangle}
\newcommand{\appinf} {\rightarrow \infty}

\newcommand{\no}{{\nonumber}}

\def\bR{{\mathbb{R}}}
\def\bC{{\mathbb{C}}}

\DeclareMathOperator*{\minimize}{minimize}
\DeclareMathOperator*{\maximize}{maximize}
\DeclareMathOperator{\tr}{tr}
\DeclareMathOperator{\diag}{diag}
\DeclareMathOperator{\vect}{vect}
\DeclareMathOperator{\tvec}{vec}
\DeclareMathOperator{\minnz}{\textnormal{min-nz}}
\def\subjectto{\mbox{subject to}}

\newcommand{\shortpar}{{\mkern3mu\vphantom{\perp}\vrule depth 0pt\mkern2mu\vrule depth 0pt\mkern3mu}}

\newenvironment{proof*}{\par\noindent{\bf Proof\ }}{}

\usepackage[colorlinks,
            linkcolor=blue,
            citecolor=blue,
            urlcolor=magenta,
            linktocpage,
            plainpages=false]{hyperref}
\usepackage[makeroom]{cancel}

\ShortHeadings{The Mixing method}{Wang, Chang, and Kolter}
\firstpageno{1}

\def\R{{\mathbb R}}

\def\minnz{{{\min}{\textnormal{-nz}}}}

\begin{document}

\author{\name Po-Wei Wang 
  \email poweiw@cs.cmu.edu \\
  \addr Machine Learning Department\\
  Carnegie Mellon University\\
  Pittsburgh, PA 15213 
  \AND
  \name Wei-Cheng Chang 
  \email wchang2@cs.cmu.edu \\
  \addr Language Technologies Institute\\
  Carnegie Mellon University\\
  Pittsburgh, PA 15213 
  \AND
  \name J. Zico Kolter
  \email zkolter@cs.cmu.edu \\
  \addr Computer Science Department\\
  Carnegie Mellon University\\
  Pittsburgh, PA 15213, and\\
  \addr Bosch Center for Artificial Intelligence\\
  Pittsburgh, PA 15222
}

\editor{}

\maketitle

\begin{abstract} In this paper, we propose a low-rank coordinate descent
approach to structured semidefinite programming with diagonal constraints.  The
approach, which we call the Mixing method, is extremely simple to implement, has
no free parameters, and typically attains an order of magnitude or better
improvement in optimization performance over the current state of the art.  
We show that the method is strictly decreasing, will converge to a critical point,
and further that for sufficient rank all non-optimal critical points are unstable. 
Moreover, we prove that with a step size, the Mixing method 
converges to the global optimum of the semidefinite program almost surely
in a locally linear rate under random initialization.
This is the first low-rank semidefinite programming method that
has been shown to achieve a global optimum on the spherical manifold without assumption. 
We apply our algorithm to two related domains:
solving the maximum cut semidefinite relaxation, and solving a maximum
satisfiability relaxation (we also briefly consider additional applications such
as learning word embeddings).  In all settings, we demonstrate substantial improvement over
the existing state of the art along various dimensions, and in total, this work
expands the scope and scale of problems that can be solved using
semidefinite programming methods.

\end{abstract}


\section{Introduction}
This paper considers the solution of large-scale, structured semidefinite
programming problems (SDPs).
A generic SDP can be written as the optimization problem
\begin{equation}
    \minimize_{X\succeq 0} \;\; \mdot{C}{X}, \;\; \subjectto \;\; \mdot{A_i}{X} = b_i, \;\; i=1\ldots p \\
  \label{eq:sdp}
\end{equation}
where $X \in \mathbb{S}^n$ is the optimization variable (a symmetric $n \times
n$ matrix), and $A_i \in \mathbb{R}^{n \times n}, b_i \in \mathbb{R}$,
$i=1,\ldots,p$ are problem data.
Semidefinite programs can encode a huge range of practical problems,
including relaxations of many combinatorial optimization tasks \citep{boyd2004convex}, approximate
probabilistic inference \citep{jordan2004semidefinite}, metric learning
\citep{yang2006distance}, matrix completion \citep{candes2012exact}, and many others.  
Unfortunately, generic semidefinite programs involve optimizing a
\emph{matrix-valued} variable $X$, 
where the number of
variables grows quadratically so that
it quickly becomes unaffordable for solvers
employing exact methods such as primal-dual interior point algorithms.

Fortunately, a property
of these problems, which has been recognized for some time now, is that the
solution to such problems is often \emph{low-rank}; specifically, the problem
always admits an optimal solution with at most rank $\lceil \sqrt{2p} \rceil$
\citep{barvinok1995problems,pataki1998rank}, and many 
SDPs are set up to have even lower rank solutions in practice.  This has
motivated the development of non-convex low-rank solvers for these systems: that
is, we can attempt to solve the equivalent (but now non-convex) formulation of
the problem
\begin{equation*}
    \minimize_{V \in \mathbb{R}^{k \times n}} \;\; \mdot{C}{V^T V}, \;\;\subjectto \;\; \mdot{A_i}{V^T V} = b_i, \;\; i=1\ldots p 
\end{equation*}
with the optimization variable $V \in \mathbb{R}^{k \times n}$.
Here we are explicitly representing $X$ by the matrix $V$ of rank $k$ (typically with $k
\ll n$), $X= V^TV$. Note that because we are representing $X$ in this way, we
no longer need to explicitly enforce semidefiniteness, as it is implied by the
change of variables.
In a long series of work dating back several years, it has been shown that,
somewhat surprisingly, this change to a non-convex problem does not cause as
many difficulties as might be thought: in practice, local solutions to the
problem tend to recover the optimal solution \citep{burer2003nonlinear}; 
assuming sufficient rank $k$,
all second order local optima of the problem are also
global optima \citep{boumal2016non}; and it even holds
that approximated local optima also have good approximation properties 
\citep{mei2017solving} for convex relaxations of some combinatorial problems.
Despite these advances, solving for $V$ remains a practical challenge. Traditional methods for handling non-linear equality constraints, such as
augmented Lagrangian methods and Riemannian manifold methods, suffer
from slow convergence, difficulties in selecting step size, or other deficiencies.

In this paper, we present a low-rank coordinate descent  approach to solving
SDPs that have the additional specific structure of constraints (only) on
the \emph{diagonal} entries of the matrix (we consider the case of unit diagonal
constraints, though we can easily extend to arbitrary positive numbers)
\begin{equation} \minimize_{X \succeq 0}\; \mdot{C}{X}, \;\; \subjectto\;\;X_{ii}=1,\;\;i=1\ldots n.
\label{eq:unitsdp}
\end{equation}
This is clearly a very special case of the full semidefinite program, but it
also captures some fundamental problems, such as the famous semidefinite
relaxation of the maximum cut (MAXCUT) combinatorial optimization problem;
indeed, the MAXCUT relaxation will be one of the primary applications in this
paper.  In this setting, if we consider the above low-rank form of the problem,
we show that we can derive the coordinate descent
updates in a very simple closed form, resulting in an algorithm several times
faster than the existing state of the art.  We call our approach the Mixing
method, since the updates have a natural interpretation in terms of giving each
$v_i$ as a mixture of the remaining $v_j$ terms.  We will also show, however,
that the method can be applied to other problems as well, such as a (novel, to
the best of our knowledge) relaxation of the MAXSAT problem.

On the theoretical side, we prove several strong properties about the Mixing
method.  Most notably, despite the fact that the method is solving a non-convex
formulation of the original MAXCUT SDP, it nonetheless will converge to the true
\emph{global} optimum of the original SDP problem, provided the rank is
sufficient, $k > \sqrt{2n}$ (which is of course much smaller than optimizing
over the entire $n^2$ variables of the original SDP).  We prove this by first
showing that the method is strictly decreasing, and always converges to a
first-order critical point.   We then show that all non-optimal critical points 
(that is, all points which are critical point in the non-convex formulation in
$V$, but not optimal in $X$ in the original SDP), are \emph{unstable},
i.e., they are saddle points in $V$ and will be unstable under updates of
the Mixing method; this means that, in practice, the Mixing method is extremely
unlikely to converge to any non-optimal solution. However, to formally prove
that the method \emph{will} converge to the global optimum, we consider a
slightly modified ``step size" version of the algorithm, for which we can prove
formally that the method indeed converges to the global optimum in all cases
(although in practice such a step size is not needed).  Finally, for both the
traditional and step size versions of the algorithm, we show that the Mixing
method attains \emph{locally linear} convergence around the global optimum.  The
primary tools we use for these proofs require analyzing the spectrum of the
Jacobian of the Mixing method at non-optimal critical points, showing that it is always
unstable around those points; we combine these with a geometric analysis
of critical points due to \cite{boumal2016non} and a convergence proof for
coordinate gradient descent due to \cite{lee2017first} to reach our main result 
(both results require slight specialization for our case, so are included for
completeness, but the overall thrust of those supporting points are due to
these past papers). 

\subparagraph{Contributions of the current work.} In summary, the main
contributions of this work are: 1) We propose a low-rank
coordinate descent method, the Mixing method, for the diagonally constrained
SDP problem, which is extremely fast and simple to implement. 2) We prove that
despite its non-convex formulation, the method is guaranteed to converge to
global optimum of the original SDP with local linear convergence, provided that we use a small 
rank $k > \sqrt{2n}$. 3) We evaluate the proposed method on the MAXCUT SDP
relaxation, showing that it is 10-100x times faster than the other
state-of-the-art solvers and scales to millions of
variables.  4) We extend the MAX-2SAT relaxation of \citet{goemans1995improved}
to general MAXSAT problems, showing that the proposed formulation can be solved
by the Mixing method in linear time to the number of literals. Further,
experiments show that the proposed method has much better approximation ratios
than other approximation algorithms, and is even comparable to the best partial
MAXSAT solvers in some instances. 

\section{Background and related work}

\paragraph{Low-rank methods for semidefinite programming}
Given an SDP problem with $p$ constraints, it was proven by
\citet{barvinok1995problems,pataki1998rank} 
that, if the problem is solvable, it admits solutions of rank $k = \lceil\sqrt{2p}\rceil$.
That is, we have solutions satisfying $X=V^T V$, such that $V\in\bR^{k\times n}$.
Thus, if we can solve the problem in the space of $V$, we can ignore the
semidefinite constraint and also have many fewer variables. 
The idea of using this low-rank structure during optimization was first
proposed by \citet{burer2003nonlinear} in their solver SDPLR, in which they solve
the low-rank problem with L-BFGS on the extended Lagrangian problem. Since then,
many low-rank optimization algorithms have been developed. One of the most
notable is the Riemannian trust region 
method introduced by \citet{absil2009optimization}. They considered the
Riemannian manifold of low-rank structures, and extended the non-linear
conjugate gradient method to work on the manifold. Later, \citet{boumal2015low}
improved the method by including preconditioned CG; these methods are
implemented in the popular Manopt package \citep{boumal2014manopt}. 

Somewhat surprisingly, all the above low-rank methods are observed to converge to
a global optimum in practice
\citep{burer2003nonlinear,absil2009optimization,boumal2014manopt}. However, to
the best of the authors' knowledge, there is not yet a general
proof on convergence to the globally optimal solution without strong
assumptions. 
\citet[Theorem 7.4.2]{absil2009optimization} proved that their method converges
to the critical point under a sufficient decrease condition, and super-linear
convergence near an isolated local minimizer when the Riemannian Hessian is
positive definite. Nonetheless, the results do not apply to the linear objective in
our problem. 
\citet[Theorem 2]{boumal2016non} proved that, for sufficiently large
$k$, all second-order optimal solutions are globally optimal for almost all cost
matrices $C$. However, this alone does not imply $f$ converges to $f^*$.
\citet{sato2015new} proved the convergence to a critical point without
assumptions for a modified Riemannian conjugate gradient method, and
\citet[Theorem 12 and Proposition 19]{boumal2016global} proved that the Riemannian trust region method converges
to a solution with Hessian larger than $-\gamma I$ in $O(1/\gamma^3)$
and provide bounds on $f-f^*$ when $k>n$.
Recently, \citet{bhojanapalli2018smoothed} proved the connection between
$\gamma$, the norm of gradients, and $f_\mu-f_\mu^*$ for the unconstrained penalty form $f_\mu$, but
this does not state the relationship between $f-f^*$ after projecting the unconstrained solution back to the feasible space.

Other existing results on global convergence to the optimal on low-rank problems do not apply to
our problem setting. For example, \citet{bhojanapalli2016dropping} proved that
under a certain spectral initialization, the gradient descent method converges to
the global optima for unconstrained low-rank SDP. \citet{park2016provable}
further proved that the method works for norm-constrained low-rank SDP problems
when the feasible space of $V$ is convex. \citet{lee2016gradient,lee2017first,o2017behavior} proved that,
under random initialization, first-order methods converge to local
minimizers for unconstrained problems. 
However, \citet{o2017behavior} only concerns the unconstrained optimization in Euclidean space,
and results in \citet{lee2016gradient,lee2017first} do not work in the spherical manifold
due to the singularity in the Jacobian (even with step size).
Because of these issues, certifying global convergence for low-rank SDP methods
typically requires a two-stage algorithm, where one 
iteratively solves the low-rank problem via some ``inner'' optimization, then checks
for a certificate of (global) optimality via the dual objective of the original
SDP and inflates the rank if the solution is not optimal (i.e., the ``outer'' 
iteration) \citep{burer2003nonlinear}. Even here, a globally optimal
solution is not theoretically guaranteed to be achieved with reduced rank in all
cases, but at least can be verified after the fact. Computing this dual
solution, however, typically requires solving an eigenvalue problem of the size of
the original $X$ matrix, so is often avoided entirely in practice in favor of
just solving one 
low-rank problem, and this is the setting we consider here.  

\paragraph{Approximation algorithms for MAXCUT and MAXSAT}
Semidefinite programming has many applications in approximation algorithms for
NP-complete problems. 
In particular, \citet{goemans1995improved} proposed a classical SDP relaxation 
on the MAXCUT and MAX-2SAT problems, which has a 0.878 approximation guarantee.
Experiments on MAX-2SAT \citep{gomes2006power} show that the SDP upper bound
and lower bound are much tighter than the classical linear programming
relaxation for MAXSAT \citep{goemans1994new}. 
While traditionally SDPs are more expensive to solve than linear
programming, we will show here that with our approach, SDP relaxations
can achieve substantially better results than linear programming in less time.

\section{The Mixing method}\label{sec:Mixing}

As mentioned above, the goal of the Mixing method is to solve the semidefinite
program \eqref{eq:unitsdp} with a unit diagonal constraint.  
As discussed, we can replace the $X\succeq 0$ constraint with
$X=V^T V$ for some $V\in\bR^{k\times n}$; when we do this, the constraint 
$X_{ii}=1$ translates to the constraint $\norm{v_i}=1$, where $v_i$ is the $i$th
column of $V$.  This leads to the equivalent (non-convex) problem  on the spherical manifold
\begin{equation}
  \label{eq:vec-programming}
  \minimize_{V\in\bR^{k\times n}}\; \mdot{C}{V^T
V}\quad\subjectto\; \norm{v_i} = 1,\;i=1\ldots n.
\end{equation}
Although the problem is non-convex, it is known
\citep{barvinok1995problems,pataki1998rank} that when $k>\sqrt{2n}$, the
optimal solution for $V\in\bR^{k\times n}$ can recover the optimal solution for $X$.

We consider solving the problem \eqref{eq:vec-programming} via
a coordinate descent method.  The resulting algorithm is extremely simple to
implement but, as we will show, it performs substantially better than existing
approaches for the semidefinite problems of interest.
Specifically, the objective terms that depend on $v_i$ are given by
$v_i^T(\sum_{j=1}^n c_{ij} v_j)$.  However, because 
$\norm{v_i}=1$ we can assume that $c_{ii}=0$ without affecting the solution of
the optimization problem.  Thus, the problem is equivalent to simply minimizing
the inner product $v_i^Tg_i$ (where $g_i=\sum_{j=1}^n c_{ij}v_j$), subject to the
constraint that $\norm{v_i}=1$; this problem has a closed form solution,
given by $v_i = - g_i / \norm{g_i}$ when $\norm{g_i}\neq 0$,
and no update otherwise.
Put in terms of the original $v_j$ variable,
the update is simply
\begin{equation*} v_i := \text{normalize}\left(-\sum_{j=1}^n
        c_{ij} v_j\right)\text{ if the norm is non-zero.}
\end{equation*}

This way, we can initialize $v_i$ on the unit sphere and perform
cyclic updates over all the $i=1\ldots n$ in closed-form. We call this the
Mixing method, because for each $v_i$ our update mixes and normalizes the remaining
vectors $v_j$ according to weight $c_{ij}$. In the case of sparse $C$
(which is common for any large data problem), the time
complexity for updating all variables once is $O(k \cdot m)$, where $k$ is 
the rank of $V$ and $m$ is the number of nonzeros in $C$. 
This is significantly cheaper than the interior point method, 
which typically admits complexities cubic in $n$.
However, the details for efficient
computation differ depending on the precise nature
of the SDP, so we will describe these in more detail in the subsequent
application sections. A complete description of the generic algorithm
is shown in Algorithm~\ref{alg:mixing}.

\begin{algorithm}[t]
        \begin{algorithmic}[1]
        \State Initialize $v_i$ randomly on a unit sphere\;
        \While{not yet converged}
                \For{$i=1,\ldots,n$}
                \State \textbf{if} $\norm{\sum_{j=1}^n c_{ij}v_j}\neq 0$ \textbf{then} $v_i := \textnormal{normalize}(-\sum_{j=1}^n c_{ij}v_j)$
                \EndFor
        \EndWhile
        \end{algorithmic}
        \caption{The Mixing method for MAXCUT problem}\label{alg:mixing}
\end{algorithm}

\subsection{Convergence properties of the Mixing methods}

This section presents the theoretical analysis of the Mixing methods, which
constitutes four main properties:
\begin{itemize}
    \item We prove that the Mixing methods are strictly decreasing in objective value
            and always converge to first-order critical points in the limit.
    \item Next, we establish the linear convergence for the Mixing methods when the iterate is close enough to a critical point, \emph{regardless of the rank}. Together with the above limit property, it means that the Mixing methods converge to a critical point in a asymptotic linear rate.
    \item Further, we prove that 
        for a rank\footnote{The tightness of the $\sqrt{2n}$ rank (actually, rank satisfying $k(k+1)/2 \geq n$)
        is proved in \cite{barvinok2001remark}.} $k>\sqrt{2n}$, 
                all non-optimal critical points $V\in\bR^{k\times n}$ are unstable for the Mixing method.
        That is, when $V^T V$ is non-optimal for the convex problem \eqref{eq:unitsdp},
                the critical point $V$ will sit on a saddle of the non-convex problem \eqref{eq:vec-programming} and the Mixing method tends to \emph{diverge} from the point \emph{locally}. 
        \item Finally, to rigorously prove the \emph{global convergence}, we show that a variant of the Mixing method with a proper step size
            converges to a global optimum almost surely under random initialization for almost every cost matrix $C$, \emph{without any assumptions}.
\end{itemize}
In total, our method represents the 
\emph{first low-rank semidefinite programming method which will
provably converge to a global optimum under constraints, and further in a locally linear rate.}

\subparagraph{Convergence to critical points.}  Our first property shows that
the Mixing method strictly decreases, and always converges to a first-order
critical point for any $k$.  This is a relatively weak statement, but useful in the context
of the further proofs.
\begin{theorem} \label{thm:critical mixing}
The Mixing method on the SDP problem \eqref{eq:unitsdp} is strictly decreasing
  and converges to first-order critical points, with step size (no assumption) and without step size (with Assumption~\ref{assume:degenerate}).
\end{theorem}
For the proof on the mixing method without step size, we introduce an additional assumption on the non-degeneracy for its normalization.
\begin{assumption} \label{assume:degenerate}
Assume that for all $i=1\ldots n$, $\norm{\sum_{j=1}^n c_{ij}v_j}$ do not
degenerate in the procedure. That is, all norms are always greater than or equal to a constant
$\delta>0$. 
\end{assumption}
In practice, the degeneracy is never observed.
The proof of Theorem~\ref{thm:critical mixing} is in Appendix~\ref{sec:critical}, which mainly involves 
setting up the Mixing iteration in a matrix form, and showing that the
difference in objective value between two iterations is given by a particular
positive term based upon this form.

\subparagraph{Locally linear convergence.}
Next, we show that the convergence of the Mixing
methods exhibits linear convergence to the global optimum whenever the solution is close enough,
regardless of the rank and the existence of nearby non-optimal critical points, 
for both versions with or without the step size.
This also echoes practical experience,
where the Mixing method does exhibit this rate of convergence.
\begin{theorem}\label{thm:Linear mixing} The Mixing methods converge linearly to the global optimum when close enough to the solution,
        with step size (no assumption) or without step size (under Assumption~\ref{assume:degenerate}).
\end{theorem}

The full proof is provided in Appendix~\ref{sec:Linear}.  We prove it from the Lipschitz smoothness (Lipschitz continuous gradient) of the
Mixing mappings.  The main difficulty here is that the corresponding linear system in the Gauss-Seidel method, $S^*\in\bR^{n\times n}$,
is semidefinite so that 
the corresponding Jacobian $J_{GS}$ contains eigenvectors of magnitude $1$ in the null space of $S^*$. 
We overcome the difficulty by proving the result directly in the function value
space like \cite{wang2014iteration} so that the eigenvectors in $\textnormal{null}(S^*)$  can be ignored.
This is the first local linear convergence on the spherical manifold without assumption.

\subparagraph{Instability of non-optimal critical points.} Next we prove the main
result of our approach, that not only is the function decreasing, but that every
non-optimal critical point is unstable; that is, although the problem \eqref{eq:vec-programming} is non-convex,
the algorithm tends to diverge (locally) from any solution that is not globally optimal.
Further, the local divergence from non-optimal critical points 
and the global convergence to critical points hint that the Mixing
method can only converges to global optimal solutions.

\begin{theorem} \label{thm:unstable}
        Pick $k>\sqrt{2n}$. For
        almost all $C$, all non-optimal first-order critical points are unstable
        fixed points for the Mixing method.
\end{theorem}
The full proof is provided in Section~\ref{sec:unstable}.  The main idea is to
show that the maximum
eigenvalue of the dynamics Jacobian, evaluated at a critical point but when $V$
is not optimal, is guaranteed to
have a spectral radius (magnitude of the largest eigenvalue) greater than
one.  We do this by showing that the Jacobian of the Mixing method has the same structure 
as the Jacobian of a standard Gauss-Seidel update plus an additional
projection matrix.  By a property from \cite{boumal2016non}, plus an
analysis of the eigenvector of Kronecker products, we can then guarantee that the
eigenvalues of the Mixing method Jacobian contains those of the
Gauss-Seidel update.  We then use an argument based upon \cite{lee2017first} to show
that the Gauss-Seidel Jacobian is similarly unstable around non-optimal critical
points, proving our main theorem.

\subparagraph{Globally optimal convergence of Mixing method with a step size.} 
Though the above two theorems in practice ensure convergence of the Mixing
method to a global optimum, because the method makes discrete updates, there is
the theoretical possibility that the method will ``jump" directly to a
non-optimal critical point (yet again, this has never been observed
in practice).  For completeness, however, in the next theorem we highlight the
fact that a version of the Mixing method that is modified with a step size 
\emph{will} always converge to a global optimum.

\begin{theorem} \label{thm:globalConvStepsize} Consider the Mixing method with a
        step size $\theta>0$. That is,
        \begin{equation*}
                v_i :=
                \textnormal{normalize}\left(v_i-\theta\sum_{j=1}^n
                c_{ij}v_j\right),\;\text{for }i=1,\ldots,n.
        \end{equation*}
        Take $k>\sqrt{2n}$ and $\theta\in(0,\frac{1}{\max_i\norm{c_i}_1})$, where $\norm{\cdot}_1$ denotes the $1$-norm.
	Then for almost every $C$, the method 
        converges to a global optimum almost surely 
        under random initialization.\footnote{Any distribution will suffice if it maps zero volume in the spherical manifold to zero probability. 
        For example, both spherical Gaussian distribution and normalized uniform distribution work.}
\end{theorem}

The full proof is provided in Appendix~\ref{sec:globalConvStepsize}. The main difference in the proof
from the step size free version is that, with a step size, we can prove
the diffeomorphism of the Mixing method and thus are able to use the stable
manifold theorem. Because the Jacobian of \emph{any} feasible method on the
spherical manifold is singular\footnote{Note that on the spherical manifold, the Jacobian of any feasible method is singular
because the Jacobian of $v_i/\norm{v_i}$ is singular.
Thus, the proof in \citet[Section 5.5]{lee2017first} does not
carry over to the case of the Mixing method, even with a step size.  This past
proof required a non-singular Jacobian, and thus different techniques are
required in our setting.}, we need to construct the inverse function
explicitly and show the smoothness of such function.  We can then use an
analysis of the Gauss-Seidel method with a step size to show our result.  To
the best of our knowledge, this represents the first globally optimal
convergence result for the low-rank method applied to (constrained) semidefinite
programming, without additional assumptions such as the Cauchy decrease 
\citep[Assumption A3]{boumal2016global},
 or solving a sequence of ``bounded'' log-barrier subproblems exactly \citep[Theorem 5.3]{burer2005local}.

\begin{remark}
Assume there are $m$ nonzeros in $C\in\bR^{n\times n}$.
With Theorem~\ref{thm:globalConvStepsize} and \ref{thm:Linear mixing},
for almost every $C$, the Mixing method with a step size admits an asymptotic complexity of 
$O(m\sqrt{n}\log\frac{1}{\epsilon})$ to reach the global optimality gap of $f-f^*\leq\epsilon$ almost surely
under random initialization. This concludes the theoretical analysis of the Mixing method.
\end{remark}

Now we will prove one of our main result: the instability of non-optimal critical points.
\subsection{Proof of Theorem~\ref{thm:unstable}: The instability of non-optimal criticals points}\label{sec:unstable}

Before starting the proofs, we discuss our notations and reformulate the Mixing methods.
\paragraph{Notations.}
We use upper-case letters for matrix, and lower-case letters for vector and scalar.
For a matrix $V\in \bR^{k\times n}$, $v_i\in\bR^{k}$ refers to the $i$-th column of $V$,
$\tvec(V)$ and $\vect(V)\in\bR^{nk}$ denote the vector stacking columns and rows of $V$, respectively. 
For a vector $y\in\bR^n$, $D_y=\diag(y)\in\bR^{n\times n}$ denotes the matrix with $y$ on the diagonal,
$y_{\max},y_{\min},y_{\minnz}\in\bR$ the maximum, the minimum, and the minimum nonzero element of $y$, respectively.
The symbol $\otimes$ denotes the Kronecker product, $\dagger$ the Moore-Penrose inverse, 
$\sigma(\cdot)$ the vector of eigenvalues,
$\rho(\cdot)$ the spectral radius,
$1_n$ the $1$-vector of length $n$,
$I_n$ the identity matrix of size $n$,
$\tr(\cdot)$ the trace,
$\mdot{A}{B}=\tr(A^T B)$ the dot product,
and $\norm{V}=\sqrt{\tr(VV^T)}$ the generalized L2 norm.
Indices $i,j$ are reserved for matrix element, and index $r$ for iterations.

\paragraph{The Mixing methods.}
We denote $C\in\bR^{n\times n}$ the cost matrix of the problem
\begin{equation*}
        \minimize_{V\in\bR^{k\times n}}\; f(V)\equiv\mdot{C}{V^TV}\quad \subjectto\;\norm{v_i}=1,\;i=1\ldots n,
\end{equation*}
and w.l.o.g.\ assume that $c_{ii}=0$ in all proofs.
Matrices $V$ and $\hat{V}$ refer to the current and the next iterate,
and $V^*$ the global optimum attaining an optimal value $f^*$
in the semidefinite program \eqref{eq:unitsdp}.\footnote{By \cite{pataki1998rank},
the optimality in \eqref{eq:unitsdp} is always attainable by $V\in\bR^{k\times n}$ when $k>\sqrt{2n}$.}
Let
\begin{equation}\label{eq:g}
        g_i=\sum_{j<i} c_{ij}\hat{v}_j +\sum_{j>i} c_{ij}v_j
\end{equation}
and matrix $L$ be the strict lower triangular part of $C$.
With these notations, the mapping of the Mixing method $M:\bR^{k\times n}\rightarrow\bR^{k\times n}$ 
and its variant $M_\theta$ with step size $\theta$ can be written as\footnote{The reason to reformulate here is to avoid the ``overwrite'' of variables in the algorithmic definition.
Moving the inverse term to the left-hand side, the reader can recover the original sequential algorithm.}
\begin{align} 
        M(V)^T&= -(L+D_y)^{-1}L^TV^T, \qquad\quad\; \text{where }y_i=\norm{g_i},\;\;\quad\quad\; i=1\ldots n,\quad\text{and}
        \label{eq:M}\\
        M_\theta(V)^T&= (\theta L+D_y)^{-1}(I-\theta L)^TV^T, \;\; \text{where }y_i=\norm{v_i-\theta g_i},\;i=1\ldots n.
        \label{eq:MS}
\end{align}
Note that in the analysis we assume $y_i>0$ without loss of generality,
otherwise we can remove those $v_i$ from the analysis because they are not updated.
Also, both $M$ and $M_\theta$ are valid functions of $V$ because $y$ can be calculated from $V$
by the original algorithmic definitions in Section~\ref{sec:Mixing}.
This formulation is similar to the classical analysis of the Gauss-Seidel method for linear equation
in \citet{golub2012matrix}, where the difference 
is that $y$ here is not a
constant to $V$ and thus the evolution is not linear.
\subsubsection*{Proof of technical lemmas.}
We start by analyzing the Jacobian of the Mixing method.
\begin{lemma}\label{lemma:Jacobian}
Using the notation in \eqref{eq:M}, the Jacobian of the Mixing method is
\begin{equation*}
        J_V = -(PL\otimes I_k+D_y\otimes I_k)^{-1} P L^T\otimes I_k.
\end{equation*}
in which $P$ is the rejection matrix of $V$. That is,
\begin{equation*}
        P=\diag(P_1,\;\ldots,\;P_n) \in \bR^{nk\times nk},\quad\text{where }P_i = I_k-\hat{v}_i\hat{v}_i^T\in\bR^{k\times k}.
\end{equation*}
\end{lemma}
\begin{proof}
Denote $V$ and $\hat{V}$ the current and the next iterate. 
Taking total derivatives on the update of the Mixing method \eqref{eq:Mg}, we have
\begin{equation*}
	y_id\hat{v}_i = -P_idg_i = -P_i (\sum_{j<i}c_{ij}d\hat{v}_j + \sum_{j>i}c_{ij}dv_j),\quad i=1\ldots n.
\end{equation*}
Moving $d\hat{v}_j$ to the left-hand side. By the implicit function theorem, we have the Jacobian
\begin{equation*}
        J_V = -\begin{pmatrix}y_1 I_k & 0 & \ldots & 0\\
			c_{12}P_2 & y_2 I_k & \ldots & 0\\
			\ldots & \ldots & \ldots & 0\\
			c_{1n}P_n & c_{2n}P_n & \ldots & y_n I_k
	\end{pmatrix}^{-1}
	\begin{pmatrix}0 & c_{12}P_1 & \ldots & c_{1n}P_1\\
			0 & 0 & \ldots & \ldots\\
			0 & 0 & \ldots & c_{(n-1)n}P_{n-1}\\
			0 & 0 & \ldots & 0
	\end{pmatrix}.
\end{equation*}
The implicit function theorem holds here because the triangular matrix is always inversible.
Rewriting $J_V$ with Kronecker product leads to the result.
\end{proof}
Note that $V=\hat{V}$ on critical points, which is also a fixed point of the Mixing method. Now we demonstrate how to analyze the Jacobian. Remember the notation $\vect(Z)=\tvec(Z^T)$. 
This way, we have the following convenient property by the Kronecker product.
\begin{lemma}\label{lemma:KroneckerRule}
        For matrices $A,B,Q,R$, we have $A\otimes B \vect(Q R^T) = \vect((AQ)(BR)^T)$.
\end{lemma}
\begin{proof}
                $A\otimes B \vect(Q R^T) = A\otimes B \tvec(RQ^T)
                = \tvec(B R Q^T A^T) = \vect((AQ)(BR)^T).$
\end{proof}
By the above property, part of the spectrum of the Jacobian can be analyzed.
\begin{lemma}[Overlapping Eigenvalues]\label{lemma:diffproj}
        Assume $V\in\bR^{k\times n}$ has $\textnormal{rank}(V)<k$. Let 
        \begin{equation*}
                P = \diag(P_1,\;\ldots,\;P_n),\quad\text{where }P_i = I_k -
                v_iv_i^T.
        \end{equation*}
        For any matrices $A,B,D\in\bR^{n\times n}$ where $(A+D)$ is inversible,
        any eigenvalue of $(A+D)^{-1}B$ is also an eigenvalue of
        \begin{equation*}
                J=[PA\otimes I_k + D\otimes I_k]^{-1}\otimes I_k P B\otimes I_k.
        \end{equation*}
\end{lemma}
\begin{proof}
        Because $\textnormal{rank}(V)<k$, by linear dependency there is a nonzero $z\in\bR^k$ such that
        \begin{equation}
                z^Tv_i = 0\;\text{ for }i=1\ldots n\quad\implies\quad P_i z = z\;\text{ for } i=1\ldots n.\no
        \end{equation}
        Observe the following matrix identity by separating the subspaces of $P$ and $I_{nk}-P$.
        \begin{align*}
                &[PA\otimes I_k + D\otimes I_k]\;\;(A+D)^{-1}\otimes I_k P\\
                ={ }&[P (A+D)\otimes I_k + (I_{nk}-P) D\otimes I_k] \;\;(A+D)^{-1}\otimes I_k P\\
                ={ }&P + (I_{nk}-P)(D(A+D)^{-1})\otimes I_k P.
        \end{align*}
        Inversing the square brackets and moving the $(I_{nk}-P)$ term to the other side, we get
        \begin{equation}
                [PA\otimes I_k+D\otimes I_k]^{-1} P
                = (A+D)^{-1}\otimes I_k P - [PA\otimes I_k+D\otimes I_k]^{-1}(I_{nk}-P)(D(A+D)^{-1})\otimes I_k P.
                \label{eq:jacob decomp}
        \end{equation}
        Beacuse there are both $P$ and $(I_{nk}-P)$ in the last term, for any vector $q$ we have from Lemma~\ref{lemma:KroneckerRule}
        \begin{align}
                &[PA\otimes I_k+D\otimes I_k]^{-1}(I_{nk}-P)(D(A+D)^{-1})\otimes I_k P \vect(qz^T)\no\\
                ={ }&[PA\otimes I_k+D\otimes I_k]^{-1}\vect\big(D(A+D)^{-1}q\;((I_k-P_i)P_i z)^T\big)_{i=1\ldots n}\no\\
                ={ }&0.
                \label{eq:jacob decomp zero}
        \end{align}
        Let $q\in\bC^n$ be an eigenvector of $(A+D)^{-1}B$ with eigenvalue $\lambda\in\bC$. Then
        \begin{align*}
                J\vect(qz^T) &= [PA\otimes I_k + D\otimes I_k]^{-1}\otimes I_k P\vect((Bq)z^T)\\
                &=((A+D)\otimes I_k)^{-1} P\vect((Bq)z^T)\\
                &= \vect((A+D)^{-1}Bq\;z^T)\\
                &= \lambda \vect(q z^T),
        \end{align*}
        where the first equality follows from Lemma~\ref{lemma:KroneckerRule},
        the second equality follows from $P_iz=z,\;\forall i$ and \eqref{eq:jacob decomp}--\eqref{eq:jacob decomp zero},
        and the last equality follows from $q$ being an eigenvector.
        Thus, every eigenvalue $\lambda$ of $(A+D)^{-1}B$ is also an eigenvalue of $J$.
\end{proof}
By the above lemma, the spectral radius of $J=-(PL\otimes I_k+D_y\otimes I_k)^{-1} P L^T\otimes I_k$ is lower bounded by $J_{GS}=-(L+D_y)^{-1}L^T$,
which can be again lower bounded as follows.
\begin{lemma}\label{lemma:GSspectral}
        For a positive vector $y\in\bR^n$, consider a matrix under the notation in \eqref{eq:M}
        \begin{equation*}
                J_{GS} = -(L+D_{y})^{-1}L^T.
        \end{equation*}
        Let $S=C+D_{y}$. 
        When $S\not \succeq 0$, the spectral radius $\rho(J_{GS})>1$.\footnote{If
        $S\succeq 0$, we can prove that the spectral radius $\rho(J_{GS})\leq 1$, in which all eigenvectors with magnitude $1$ reside in the null of $S$,
        as an immediate result from \citet[Corollary 3.4]{wang2014iteration}. However, the result is not used here.}
\end{lemma}
\begin{proof}
The proof is more technical and is given in Appendix~\ref{sec:GSspectral}.
\end{proof}
Further, the assumption in Lemma~\ref{lemma:diffproj} is fulfilled by the following property of critical points.
\begin{lemma}\label{lemma:deficient}\citep[Lemma 9]{boumal2016non}
        Let $\frac{k(k+1)}{2}>n$. Then, for almost all $C\in\bR^{n\times n}$, all first-order critical points $V\in\bR^{k\times n}$ have rank smaller than $k$.
\end{lemma}
\begin{proof} The proof is listed in Appendix~\ref{sec:defficient} for completeness.
\end{proof}
Next, we characterize the optimality of $V$ by proving all non-optimal $V$ admits an $S\not\succeq 0$.
\begin{lemma}\label{lemma:obvious}
        For a critical solution $V$, denote $S=C+\diag(y)$, where $y_i=\norm{Vc_i}, \;\forall i$.
        Then
        \begin{equation*}
                S\succeq 0 \;\text{ if and only if }\; V\text{ is optimal.}
        \end{equation*}
        Further, if $V$ is optimal, all $y_{i}$ are positive except when $c_i=0$.\footnote{Let $y_i^*=\norm{V^*c_i}$ and $S^*=C+\diag(y^*)\succeq 0$. An immediate consequence of the lemma is that, for any feasible $U$,
        $f(U)-f^* = \tr(UCU^T)+1_n^T y^* = \tr(US^*U^T)$. Further, suppose $U$ is also an optimum, then $f(U)-f^* =\tr(US^*U^T)=0 \iff US^*=0\iff \norm{Uc_i}=y_i^*=\norm{V^*c_i},\;\forall i.$ That is, $y^*$ is unique.}
\end{lemma}
\begin{proof}
        Consider the dual problem of the SDP problem \eqref{eq:unitsdp},
        \begin{equation*}
                \maximize_{y\in\bR^n}\; -1_n^Ty,\quad\subjectto\;\;C+\diag(y)\succeq 0.
        \end{equation*}
        If $S=C+\diag(y)\succeq 0$, variable $y$ becomes a feasible solution of the above dual problem.
        Further, since $V$ is a critical solution, we have 
        \begin{equation*}
        VS=0\;\implies\;V^TVS=0\;\implies\; \tr(V^TVC)=-\tr(V^TV\diag(y))=-1_n^Ty,
        \end{equation*}
        which means that $V^TV$ and $y$ are primal and dual optimal solutions that close the duality gap.
        Thus, for critical $V$, $S\succeq 0$ implies optimality, and non-optimality implies $S\not\succeq 0$.
        
        On the other direction, when the solution $V$ is optimal, there will be a corresponding dual optimal solution $y$ satisfying
        \begin{equation*}
                V^TV(C+\diag(y)) = 0\implies v_i^T V(C+\diag(y))=0,\;\forall i\implies y_i = \norm{Vc_i},\;\forall i.
        \end{equation*}
        And $S=C+\diag(y)\succeq 0$ follows from the dual feasibility. 
        By the characterization of SPSD matrix, all submatrix of $S\succeq 0$ are SPSD. 
        Thus, $y_{i}\geq 0$. 
        If equality $y_{i}=0$ holds, by the same reason all $2\times 2$ submatrix 
        $\begin{pmatrix}0&c_{ij}\\c_{ij}&y_{jj}\end{pmatrix}\succeq 0$, $\forall j$.
        This means $c_i=0$.
\end{proof}

\subsubsection*{Proof of Theorem~\ref{thm:unstable}}
\begin{proof}
We first derive the Jacobian $J$ of the Mixing method in Lemma~\ref{lemma:Jacobian}, which gives
\begin{equation*}
  J = -(PL\otimes I_k+D_y\otimes I_k)^{-1} P L^T\otimes I_k,
\end{equation*}
where $P=\diag(P_1,\ldots,P_n)$ and $P_i=I-v_iv_i^T$ because $\hat{v}_i=v_i$ on the critical point (also a fixed point).
In Lemma~\ref{lemma:diffproj}, we prove that when $\textnormal{rank}(V)<k$, 
the eigenvalues of $J$ contain the eigenvalues of 
\begin{equation*}
  J_{GS} = -(L+D_y)^{-1}L^T.
\end{equation*}
The assumption in Lemma~\ref{lemma:diffproj} is fulfilled by Lemma~\ref{lemma:deficient},
which guarantees that for almost every $C$, all the first-order critical point must have $\textnormal{rank}(V)<k$.
Further, Lemma~\ref{lemma:GSspectral} and \ref{lemma:obvious}
show that $J_{GS}$ happens to be the Jacobian of the Gauss-Seidel method on a linear system, 
which has a spectral radius $\rho(J_{GS})>1$ on the non-optimal first-order critical point $V$.
Thus, Lemma~\ref{lemma:diffproj} implies $\rho(J)\geq\rho(J_{GS})>1$, which means that all non-optimal first-order critical points are unstable for the Mixing method.
\end{proof}

\section{Applications}

\subsection{Maximum cut problem}\label{sec:maxcut}
The SDP MAXCUT relaxation is indeed
the motivating example of the Mixing method, so we consider it first.  In this
section, we demonstrate how to apply our method to this problem, which
originated from \citet{goemans1995improved}.

\paragraph{Problem description. }
The maximum cut problem is an
NP-hard binary optimization problem, which seeks a partition over a set of vertices
$i=1\ldots n$, so that the sum of edge weights $c_{ij}$ across the partition is
maximized. If we denote the two partitions as $\pm 1$, we can formulate the
assignment $v_i$ of vertex $i$ as the following binary optimization problem
\begin{equation*} 
        \maximize_{v_i\in\{\pm 1\},\;\forall i}\;\; \frac{1}{2}\sum_{ij}
c_{ij}\left(\frac{1-v_iv_j}{2}\right).
\end{equation*} 
\citet{goemans1995improved} proposed that we can approximate the
above solution by ``lifting'' the assignment $v_i$ from $\{\pm 1\}$ to a unit sphere in $\mathbb{R}^k$ for
sufficiently large $k$ as
\begin{equation*} \maximize_{\norm{v_i}=1,\;\forall i}\;\; \frac{1}{2}\sum_{ij}
c_{ij}\left(\frac{1-v_i^T v_j}{2}\right).
\end{equation*} To recover the binary assignment, we can do a randomized rounding
by picking a random vector $r\in\bR^k$ on the unit sphere, and letting the binary
assignment of vertex $i$ be $\text{sign}(r^T v_i)$. Their analysis shows that the
approximation ratio for the NP-hard problem is $0.878$, which means that the
expected objective from the randomized rounding scheme is at least $0.878$ times
the optimal binary objective.

\paragraph{Algorithm Design.} 
Because the problem can be solved by the unit diagonal SDP \eqref{eq:unitsdp}, we can apply
the Mixing method directly, as presented in Algorithm \ref{alg:mixing}. Further,
for a sparse adjacency matrix $C$, 
the coefficient $\sum_{j}c_{ij}v_j$ can be constructed in time proportional to
the nonzeros in column $i$ of $C$. Thus, the time complexity of running a round
of
updates for all $v_i$ is $O\left(k\cdot \#\textnormal{edges}\right)$, in which
$k$ is at most $\sqrt{2n}$.

\subsection{Maximum satisfiability problem} Using similar ideas as in 
the previous section, \citet{goemans1995improved} proposed that we can use
SDP to approximate the maximum 2-satisfiability problem. In this section, we
propose a formulation that generalizes this idea to the general maximum satisfiability
problem, and apply the Mixing method to this problem.  The proposed relaxation
here is novel, to the best of our knowledge, and (as we will show) achieves
substantially better approximation results than existing relaxations. 

\paragraph{Problem description. }The MAXSAT problem is an extension of the
well-known satisfiability problem, where the goal is to find an assignment that
\emph{maximizes} the number of satisfied clauses.  Let $v_i\in\{\pm 1\}$ be
a binary variable and $s_{ij}\in\{-1, 0, 1\}$ be the sign of variable $i$ in
clause $j$.  The goal of MAXSAT can then be written as the optimization problem
\begin{equation*}
  \maximize_{v\in\{-1,1\}^n}\;\; \sum_{j=1}^m \bigvee_{i=1}^n \mathbf{1}\{s_{ij} v_i > 0 \}.
\end{equation*}
Note that most clauses will contain relatively few variables, so the $s_{j}$
vectors will be sparse.  To avoid the need for an additional bias term, we
introduce an auxiliary ``truth'' variable $v_0$, and 
define $z_j = \sum_{i=1}^ns_{ij}v_i - 1 = \sum_{i=0}^n s_{ij}v_i = Vs_j$.  Then the
MAXSAT problem can be approximated as
\begin{equation*}
        \maximize_{v\in\{-1,1\}^n}\;\;\sum_{j=1}^m 1-\frac{\norm{Vs_j}^2 - (|s_j|-1)^2}{4|s_j|}.
\end{equation*}
Although we will not derive it formally, the reader can verify that 
for any configuration $v\in\{-1,1\}^n$, this term represents an upper bound on the exact MAXSAT solution.\footnote{Actually, the formula matches the approximation of \citet{goemans1995improved} for MAX-2SAT.}
Similar to the MAXCUT SDP, we can relax the $v_i$s to be vectors in $\mathbb{R}^k$ with
$\norm{v_i}=1$.  This leads to the full MAXSAT semidefinite programming relaxation
\begin{equation*}
                \minimize_{X \succeq 0}\;\;\mdot{C}{X},\;\; \subjectto \;\;C=\sum_{j=1}^m w_j s_j s_j^T,\;\; X_{ii}=1,\;i=0\ldots n,
\end{equation*}
where $w_j=1/(4|s_j|)$.

\begin{algorithm}[t]
        \begin{algorithmic}[1]
        \State Initialize all $v_i$ randomly on a unit sphere, $i=1\ldots n$.\;
        \State Let $z_j=\sum_{i=0}^n s_{ij}v_i$ for $j=1,\ldots,m$\;
        \While{not yet converged}
                \For{$i=1,\ldots,n$}
                        \State \textbf{For each} $s_{ij}\neq 0$ \textbf{do} {$z_j := z_j - s_{ij}v_i$}
                        \State $v_i := \textnormal{normalize}\left(-\sum_{j=1}^m \frac{s_{ij}}{4|s_j|} z_j\right)$ if the norm is non-zero
                        \State \textbf{For each} $s_{ij}\neq 0$ \textbf{do} {$z_j := z_j + s_{ij}v_i$}
                \EndFor
        \EndWhile
        \end{algorithmic}
        \caption{The Mixing method for MAXSAT problem}\label{alg:maxsat}
\end{algorithm}
\paragraph{Algorithm Design.}  Because the $C$ matrix here is not sparse ($s_j s_j^T$ has $|s_j|^2$ non-sparse entries), we need a
slightly more involved approach than for MAXCUT, but the algorithm is still extremely simple.
Specifically, we maintain $z_j=Vs_j$ for all clauses $j$. 
Because in each subproblem only one variable $v_i$ is changed, $z_j$ can be
maintained in $O(k \cdot m_i)$ time, where $m_i$ denotes the number of
clauses that contain variable $i$. In total,
the iteration time complexity is $O(k \cdot m)$,
where $m$ is the number of literals in the problem.
Also, because applying arbitrary rotations $R\in\bR^{k\times k}$ to $V$ does not
change the objective value of our problem, we can avoid updating $v_0$. 
Algorithm~\ref{alg:maxsat} shows the complete algorithm.
To recover the binary assignment, we apply the following classic rounding
scheme: sample a random vector $r$ from a unit sphere, then assign binary
variable $i$ as true if $\text{sign}(r^T v_i)=\text{sign}(r^T v_0)$ and false
otherwise.

\section{Experimental results}

\begin{figure*}[t] \centering
\includegraphics[scale=0.36]{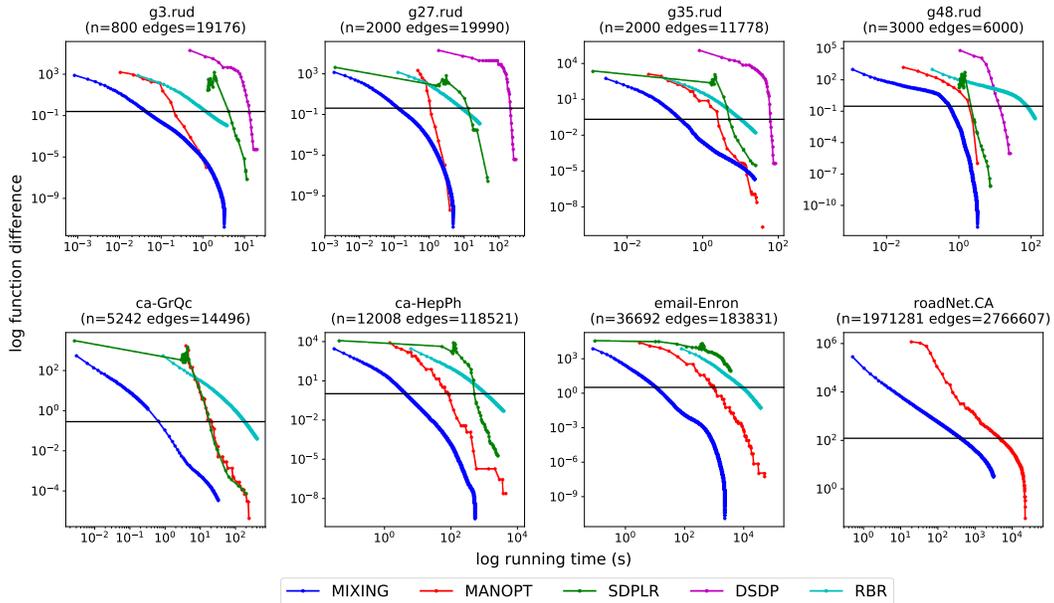}

\caption{Objective value difference versus training time for the MAXCUT problems
  (log-log plot, lower is better). The horizontal lines mark the default stopping precision of
  the Mixing method, which is $10^{-4}$ times the starting relative objective of
  the Mixing method. Experiments show that our method (the blue line) is
  10-100x faster than other solvers on our default stopping precision.
  Note that sometimes curves for SDPLR, DSDP, and RBR are not plotted because they either crashed or did not output any solution after an hour.} 
  \label{fig:maxcut-fval}
\end{figure*}

\paragraph{Running time comparison for MAXCUT}
Figure \ref{fig:maxcut-fval} shows the results of running the Mixing method on
several instances of benchmark MAXCUT problems.  These range in size from
approximately 1000 nodes and 20000 edges to approximately 2 million nodes and 3
million edges.  For this application, we are largely concerned with evaluating
the runtime of our Mixing method versus other approaches for solving the same
semidefinite program.  Specifically, we compare to DSDP \citep{dsdp5}, a mature interior
point solver; SDPLR \citep{burer2003nonlinear}, one of the first approaches to use low-rank
structures; Pure-RBR \citep{wen2009row,wen2012block}, a coordinate descent method in the $X$ space, 
which outputs the best rank-1 update at each step;
and Manopt \citep{boumal2014manopt}, a recent toolkit for optimization on Riemannian
manifolds, with specialized solvers dedicated to the MAXCUT problem.\footnote{We
didn't compare to commercial software like MOSEK, an interior-point solver like DSDP, because it is not open-source and \citet{boumal2015riemannian} already showed that SDPLR is faster than MOSEK on diagonally constrained problems.}
To be specific, we use DSDP 5.8, SDPLR 1.03-beta, and Manopt 3.0 with their default parameters.
For Manopt, we compare to a subroutine "elliptopefactory”, specially designed for diagonal-constrained SDP.
For Pure-RBR, we implement the specialized algorithm \citep[Algorithm 2]{wen2009row} for MAXCUT SDP with a sparse graph in C++,
which only requires a single pass of the sparse matrix per iteration. We omit the log barrier and initialize the RBR with full-rank $X$.
All experiments are run on an Intel Xeon E5-2670 machine with 256 GB memory,
and all solvers are run in the single-core mode to ensure fair comparisons.
As the results show, in all cases the Mixing method is substantially faster than
other approaches: for reaching modest accuracy (defined as $10^{-4}$ times the
difference between the initial and optimal value), we are typically 10-100x
faster than all competing approaches; only the Manopt algorithm ever surpasses our
approach, and this happens only once both methods have achieved very high
accuracy.  Crucially, on the largest problems, we remain about 10x (or more)
faster than Manopt over the entire run, which allows the Mixing method to scale
to substantially larger problems.

\begin{figure*}[t] \centering

\includegraphics[scale=0.36]{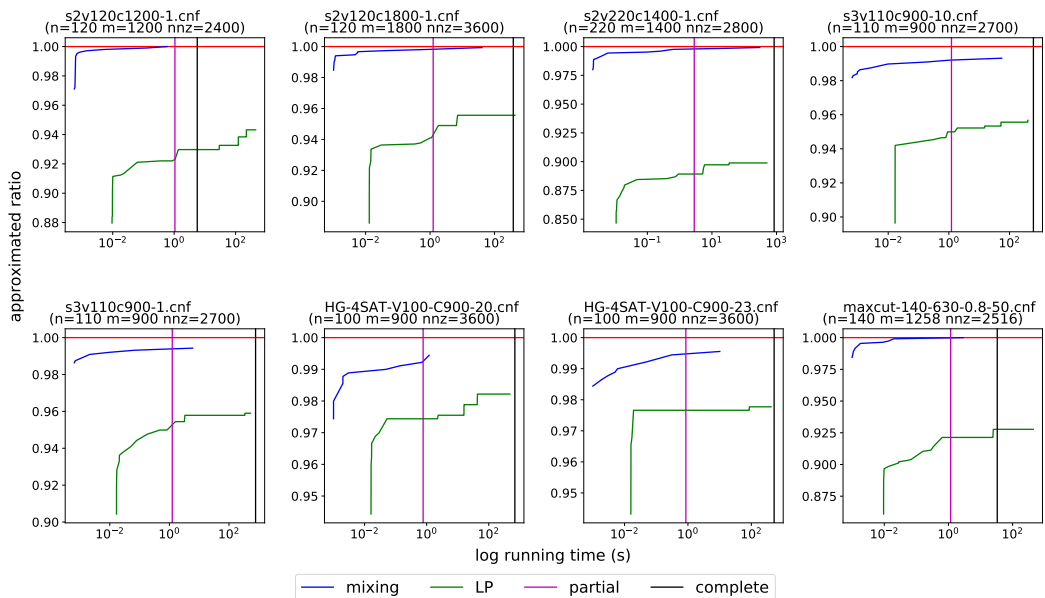}
\caption{Approximated ratio versus (log) running time for the MAXSAT problems (higher is better). 
The horizontal line marks the perfect approximation ratio ($1.00$), 
and the curves mark the approximation ratio of different approximation algorithms over time.
Experiments indicate that our proposed formulation/method (blue curves) achieves better approximation ratios in less time compared to LP.
Further, it is sometimes faster than the best partial solvers (purple vertical lines) and complete solvers (black vertical lines) in the MAXSAT 2016 competition.}
  \label{fig:maxsat-fval}
\end{figure*}

\paragraph{Effectiveness of the Mixing method on approximating MAXSAT problems.}
 Unlike the previous experiment (where the focus was solely on optimization
performance), in this section we highlight the fact that with the Mixing method
we are able to obtain MAXSAT results with a high approximation ratio on
challenging domains (as the problems are similar, relative optimization
performance is similar to that of the MAXCUT evaluations).  Specifically, we
evaluate examples from the 2016 MAXSAT competition \citep{maxsatrace2016} 
and compare our result to the best heuristic \emph{complete} and \emph{partial} solvers. Note that the complete solver produces a verified result, while the partial solver outputs a non-verified solution. 
Out of 525 problems solved in the complete track (every problem solved exactly by some solvers within 30 minutes
during the competition), our method achieves an average approximation ratio of
0.978, and usually finds such solutions within seconds or less. Further, in some instances we obtain perfect solution faster than the best partial solvers.
Figure
\ref{fig:maxsat-fval} shows the progress of the approximate quality versus the running
time. Beside the best heuristic solvers in MAXSAT 2016, we also show the approximation ratio over time for 
the well-known linear programming
approximation \citep{goemans1994new} (solved via the Gurobi solver). 
Note that each point in the blue and green curves denote the approximation ratio of the output solution at the time,
and the starting points of the curves denote the time that the solver output the first solution.
In all cases the Mixing method gives better and faster solutions than the LP approximations.

\paragraph{The sufficient rank for optimization and randomized rounding}
We proved that a low-rank of $\sqrt{2n}$ is sufficient for optimizing the SDP \citep{barvinok1995problems,pataki1998rank}.
But can we use an even lower rank? Will the low rank affect the quality of randomized rounding?
In Figure~\ref{fig:maxsat-rank}, we ran several MAXSAT SDP instances with ranks ranging from $1$ to double the theoretical rank upper-bound and record the relative error, which is $(f-f^*)/f^*$ for the objective value and $(\text{unsat}-\text{opt})/\text{\#clauses}$ for the approximation ratio of the randomized rounding.
The experiment shows that a very low rank ($\leq 5$ in those instances) is enough to achieve nearly optimal performance in both the optimization and the randomized rounding process.
Further, we see that increasing the rank beyond the theoretical upper bound doesn't improve the performance.
That is, the $\sqrt{2n}$ rank is sufficient for both optimization and randomized rounding,
and it is possible to use a lower rank without hurting the performance.

\begin{figure*}[t] \centering
\includegraphics[scale=0.43]{\detokenize{figure/all}.eps}
\caption{The relative error vs rank in the objective value and approximation (the lower, the better).
The x-axis is the ranks ranging from $1$ to $2x$ the theoretical rank upper-bound ($\sqrt{2n}$).
Experiments suggests that a very low rank ($\leq 5$) is sufficeint to achieve nearly optimal approximation error in both the optimization and the randomized rounding process. Further, the error doesn't change much after the theoretical rank upper-bound.}
  \label{fig:maxsat-rank}
\end{figure*}

\section{Conclusion}

In this paper we have presented the Mixing method, a low-rank coordinate descent approach
for solving diagonally constrained semidefinite programming problems.
The algorithm is extremely simple to implement, and involves no free parameters
such as learning rates.  
In theoretical aspects, we have proved that the method converges to a first-order critical point
and all non-optimal critical points are unstable under sufficient rank. 
With a proper step size, the method converges to the global optimum almost surely under random initialization.
This is the first convergence result to the global optimum on the spherical manifold without assumption.
Further, we have shown that the proposed methods admit local linear convergence in the neighborhood of the optimum
regardless of the rank.
In experiments, we have demonstrated the method
on three different application domains: the MAXCUT SDP, a  MAXSAT
relaxation, and a word embedding problem (in the appendix).  In all cases we show positive
results, demonstrating that the method performs much faster than existing approaches from an
optimization standpoint (for MAXCUT and word embeddings), and that the resulting
solutions have high quality from an application perspective (for MAXSAT).  In
total, the method substantially raises the bar as to what applications can be feasibly
addressed using semidefinite programming, and also advances the state of the art
in structured low-rank optimization.

\section*{Acknowledgement}
We thank Gary L. Miller for his helpful discussion on the geometry of the unit sphere,
and Simon S. Du for pointing out many recent references.

\bibliography{sdp}

\newpage
\appendix
\section{Proof of Theorem~\ref{thm:critical mixing}: Convergence to critical points}\label{sec:critical}
\begin{lemma}\label{lemma-sufficient} Let $\hat{V}=M(V)$ for the Mixing method $M:\bR^{k\times n}\rightarrow \bR^{k\times n}$ defined in \eqref{eq:g} and \eqref{eq:M}.
\begin{equation*} f(V)-f(\hat{V})
        = \sum_{i=1}^n y_i\norm{v_i-\hat{v}_i}^2.
\end{equation*}
\end{lemma}

\begin{proof} 
Recall the objective function $f$ w.r.t.\ variable $v_i$ while fixing all other variables $v_j$ is
        \begin{equation*}
                \mdot{C}{V^TV} = \sum_i\sum_j c_{ij}v_i^Tv_j = 2 v_i^T(\sum_j c_{ij}v_j)+\text{constant}.
        \end{equation*}
Note that the $\sum_j c_{ij}v_j$ term is independent of $v_i$ because $c_{ii}=0$.
Now consider the inner cyclic iteration of the Mixing method updating $v_i$ to $\hat{v}_i$.
Because only those $v_j$ with $j<i$ are updated to $\hat{v}_j$, 
the objective value before updating $v_i$ to $\hat{v}_i$ equals 
        $2v_i^T(\sum_{j<i}c_{ij}\hat{v}_j+\sum_{j>i}c_{ij}v_j) = 2g_i^Tv_i$ plus constants,
Thus, the updates of the Mixing method can be written as
\begin{equation}\label{eq:Mg}
        \hat{v}_i = -g_i/y_i,\quad\text{where }y_i=\norm{g_i}\text{ and }g_i = \sum_{j<i} c_{ij} \hat{v}_j + \sum_{j>i} c_{ij} v_j ,\quad i=1\ldots n.
\end{equation}
and the objective difference after updating $v_i$ to $\hat{v}_i$ is
        \begin{equation*}
                2g_i^T(v_i-\hat{v}_i) = -2\norm{g_i}\hat{v}_i^T(v_i-\hat{v}_i) = 2y_i (1-v_i^T\hat{v}_i) = y_i \norm{v_i-\hat{v}_i}^2.
        \end{equation*}
The result follows from summing above equation over $i=1\ldots n$.
\end{proof} 

\paragraph{Proof of Theorem~\ref{thm:critical mixing}}$ $\\
\begin{proof}
From Lemma~\ref{lemma-sufficient} and the compactness of the space of $V$,
the sequence $\{f(V^r)\}$ for iterate $V^r$ is monotonically decreasing
and converges to a value $\bar{f}$.
The function decrease converges to zero because of Lemma~\ref{lemma-sufficient} and $y_i=\norm{g_i}$ being bounded.
Further, the normalizer $y_i$ does not degenerate due to Assumption~\ref{assume:degenerate},
so we have
        \begin{equation*}
        \lim_{r\appinf} \norm{v^r_i-\hat{v}^r_i}^2 = 0.
        \end{equation*}
That is, every limit point is a fixed point.
Now we prove that every limit point in $\{V^r\}$ is also a critical point, which has zero Riemannian gradient.
The Riemannian gradient on the spheres is defined as
        \begin{equation*}
                \texttt{grad}(V)_i = (I-v_iv_i^T)\tilde{g}_i, \forall i=1,\ldots,n.
        \end{equation*}
        where $\tilde{g}_i=Vc_i$ is the gradient for coordinate block $i$. 
        Then we have
        \begin{equation*}
                \norm{\texttt{grad}(V)_i}^2 = \norm{g_i}^2-(v_i^Tg_i)^2 
                \leq 2\norm{g_i}(\norm{g_i}+v_i^Tg_i) = \norm{g_i}\cdot 2g_i^T(v_i-\hat{v}_i),
        \end{equation*}
where the inequality is from $(\norm{g_i}+v_i^Tg_i)^2\geq 0$.
Note that the RHS term is the function decrease in Lemma~\ref{lemma-sufficient}.
Align the above term to a specific iterate $V^r$ by taking out the difference $v_i^r-\hat{v}_i^r$ in $g_i$,
and taking limit to the inequality, we have    
        \begin{equation*}
                \lim_{r\appinf} \norm{\texttt{grad}(V^{r})}^2 \leq\lim_{r\appinf}\norm{g_i}\cdot 2g_i^T(v_i^r-\hat{v}_i^r) + \sum_i O(\norm{v_i^r-\hat{v}_i^r})= 0,
        \end{equation*}
That is, the Riemannian gradient converges to zero in the limit.
The step-sized version follows from the same argument.
\end{proof}

\section{Proof of Theorem~\ref{thm:Linear mixing}: Local Linear convergence}\label{sec:Linear}
In the following lemma, we use the asymptotic analysis inspired by \cite[Theorem 4]{erdogdu2018convergence},
while removing its assumption by identifying the inherent Rayleigh quotient problem.
\begin{lemma}\label{lemma:VSErrorBnd}
        Define $\hat{y}_i = v_i^TVc_i$, $\forall i=1,\ldots,n$. Let $\hat{S}=C+D_{\hat{y}}$, and let $\bar{f}$ be the objective value of the nearest first-order critical point. 
        Then, there is a neighborhood around the critical point such that
        \begin{equation*}
                \norm{V\hat{S}}^2\geq \frac{\kappa}{2}(f(V)-\bar{f})
        \end{equation*}
        for $V$ with $f(V)\geq \bar{f}$ with a constant $\kappa>0$.
        That is, $\norm{VS}^2$ is a local error bound.
\end{lemma}
\begin{proof}
        Let $\bar{V}$ be the nearest first-order critical point to $V$.
        Consider the geodesic on the spheres around $\bar{V}$ w.r.t.\ a unit tangent direction $U$ such that $u_i^T\bar{v}_i=0$ and $\norm{U}=1$, which is 
        \begin{equation*}
                v_i = \bar{v}_i \cos(\norm{u_i}t) + \frac{u_i}{\norm{u_i}}\sin(\norm{u_i}t), \forall i=1,\ldots,n.
        \end{equation*}
        Taking the Taylor expansion around $\bar{V}$ and using $\bar{V}\bar{s}_i=0$, we have
        \begin{equation*}
                v_i = \bar{v}_i + t u_i + O(t^2)\quad\text{and}\quad
                V\bar{s}_i = t U\bar{s}_i + O(t^2).
        \end{equation*}
        Substitute $\hat{S}=\bar{S}+D_{\hat{y}}-D_{\bar{y}}$ and applying the expansion and , there is
        \begin{align*}
                \norm{V\hat{S}}^2 &= \sum_i\norm{V\bar{s}_i}^2-(v_i^TV\bar{s}_i)^2=t^2 \left(\sum_i \norm{U\bar{s}_i}^2-(\bar{v}_i^TU\bar{s}_i)^2\right) + O(t^3)\\
                &= t^2 \sum_i \norm{(I-\bar{v}_i\bar{v}_i^T)U\bar{s}_i}^2+ O(t^3).
        \end{align*}
        In the other hand, applying the expansion to $f(V)-\bar{f}$, we have
        \begin{equation*}
                f(V)-\bar{f}=\tr(V^TV\bar{S}) 
                = t^2 \tr(U^TU\bar{S}) + O(t^3).
        \end{equation*}
        Note that $\tr(U^TU\bar{S})\geq 0$, otherwise it contradicts with $f(V)\geq \bar{f}$ when $t$ is small enough.
        Now consider the generalized Rayleigh quotient problem relating the $t^2$ terms in $\norm{V\hat{S}}^2$ and $f(V)-\bar{f}$.
        \begin{equation}
                \kappa := \inf_{U}\; \frac{\sum_i \norm{(I-\bar{v}_i\bar{v}_i^T)U\bar{s}_i}^2}{\tr(U^TU\bar{S})},
                \quad\text{s.t.}\;\tr(U^TU\bar{S})\geq 0,\;\;u_i^T\bar{v}_i=0,\;\forall i.
        \end{equation}
        The numerator equals $\vect(U)^T M \vect(U)$ for an SPSD matrix $M$ because it's a sum-of-squares.
        Further,
        \begin{equation*}
                (I-\bar{v}_i\bar{v}_i^T)U\bar{s}_i = 0\;\implies\; \exists \theta,\; U\bar{s_i} =\theta v_i\;\implies u_i^TU\bar{s}_i = \theta u_i^T\bar{v}_i = 0.
        \end{equation*}
        Thus, $\tr(U^TU\bar{S})$ contains all the null of $\sum_i \norm{(I-\bar{v}_i\bar{v}_i^T)U\bar{s}_i}^2$.
        So $\kappa\geq \sigma_{\minnz}(M)/\sigma_{\max}(\bar{S})>0$ because $\tr(U^TU\bar{S})\geq 0$.
        Together, there is
        \begin{equation*}
                \norm{V\hat{S}}^2 = \frac{\sum_i \norm{(I-\bar{v}_i\bar{v}_i^T)U\bar{s}_i}^2}{\tr(U^TU\bar{S})} (f(V)-\bar{f}) + O(t^3) \geq \frac{\kappa}{2} (f(V)-\bar{f})
        \end{equation*}
        in the neighborhood of $\bar{V}$ when $t$ is small enough.
\end{proof}

\begin{lemma}\label{lemma:VShat}
        Let $S=C+D_y$ for arbitrary $y$ and $\hat{S}=C+D_{\hat{y}}$ with $\hat{y}_i = -v_i^TVc_i$. Then
                $\norm{VS}^2\geq \norm{V\hat{S}}^2$.
\end{lemma}
\begin{proof}
        Note that
        \begin{equation*}
                v_i^TV\hat{s}_i = v_i^T(Vc_i - (v_i^TVc_i) v_i) = 0.
        \end{equation*}
        Using this property,
        \begin{equation*}
                \norm{VS}^2 = \norm{V\hat{S}+V(D_y-D_{\hat{y}})}^2 = \norm{V\hat{S}}^2 + \norm{D_y-D_{\hat{y}}}^2 \geq \norm{V\hat{S}}^2.
        \end{equation*}
        Thus, the result holds.
\end{proof}
\begin{lemma}\label{lemma:LipschitzM}
        Under Assumption~\ref{assume:degenerate}, the Mixing method $M:\bR^{k\times n}\rightarrow \bR^{k\times n}$ satisfies
        \begin{equation}\label{eq:LipschitzM:1}
                \norm{V-M(V)}^2 \geq \frac{1}{y_{\max}^2}\norm{V\hat{S}}^2.
        \end{equation}
\end{lemma}
\begin{proof*}
        Let $S=C+D_y$ with $y=y(V)$.
        Under the notation in \eqref{eq:M}, 
        \begin{equation}
                V-M(V) = V(L^T+D_y)(L^T+D_y)^{-1}+VL(L^T+D_y)^{-1} = VS(L^T+D_y)^{-1}\no.
        \end{equation}
        For simplicity, let $R:=(L^T+D_y)^{-1}$. Then
        \begin{equation}
                \norm{V-M(V)}^2 = \norm{VSR}^2 \geq \sigma_{\minnz}^2(R)\norm{VS}^2 \geq \sigma_{\minnz}^2(R)\norm{V\hat{S}}^2.
        \end{equation}
        The last inequality follows from Lemma~\ref{lemma:VShat}. 
        The result holds with $\sigma_{\minnz}^2(R)=1/y_{\max}^2$.
\end{proof*}

\begin{lemma}\label{lemma:LipschitzMS}
The Mixing method $M_\theta:\bR^{k\times n}\rightarrow \bR^{k\times n}$ with step size $\theta$ satisfies
        \begin{equation}\label{eq:LipschitzMS:3}
                \norm{V-M_\theta(V)}^2 \geq \frac{\theta^2}{y_{\max}^2}\norm{V\hat{S}}^2.
        \end{equation}
\end{lemma}
\begin{proof*}
        Let $S=C+\frac{1}{\theta}(D_y-I)$ because the normalizer $y_i=\norm{v_i-\theta g_i}$. By the derivation in \eqref{eq:MS}, we have
        \begin{equation*}
                V-M_\theta(V) = V-V(I_n-\theta L)(\theta L^T+D_{y})^{-1} = V(D_{y} -I_n + \theta C)(\theta L^T+D_{y})^{-1} = \theta VS(\theta L^T+D_y)^{-1}.
        \end{equation*}
        Following the same analysis in Lemma~\ref{lemma:LipschitzM}, 
        the result holds.
\end{proof*}

\paragraph{Proof of Theorem~\ref{thm:Linear mixing}} $ $\\
\begin{proof}
        By Lemma~\ref{lemma:VSErrorBnd} and Lemma~\ref{lemma:LipschitzM},
        there is a neighborhood around $\bar{V}$ such that
        \begin{equation*}
                \norm{V-M(V)}^2 \geq \frac{\kappa}{2y_{\max}^2}(f(V)-\bar{f}).
        \end{equation*}
        Together with Lemma~\ref{lemma-sufficient} (under Assumption~\ref{assume:degenerate}),
        \begin{equation}
                f(V)-f(M(V)) \geq \frac{y_{\minnz}\kappa}{2y_{\max}^2} (f(V)-f^*) \implies (1-\frac{y_{\minnz}\kappa}{2y_{\max}^2})(f(V)-f^* ) \geq f(M(V))-\bar{f}.\no
        \end{equation}
        That is, the Mixing method $M$ converges R-linearly to the $\bar{f}$ in the neighborhood of the critical point. Moreover, we know the method always reaches the neighborhood by Theorem~\ref{thm:critical mixing}.
        The same local linear convergence follows for the Mixing method $M_\theta$ with step size $\theta\in(0,\tfrac{1}{\max_i \norm{c_i}_1})$ from Lemma~\ref{lemma:squaredDecreaseMS} (no assumption) and Lemma~\ref{lemma:LipschitzMS}.

        Further, because $f(V)-\bar{f}$ is converging to zero exponentially, there is a $\delta\in(0,1)$ such that $f(V^r)-\bar{f}=O(\delta^r)$ for all large enough $r$.
        Consequently, $f(V^r)-f(M(V^r))=O(\delta^r)$ and $\norm{V-M(V)}^2=O(\delta^r)$ follows,
        so we have
        \begin{equation*}
                \norm{V^r-\bar{V}} \leq \sum_{t=r}^\infty \norm{V^t-V^{t+1}} = O(\sum_{t=r}^\infty \delta^{t/2})
                = O(\delta^{r/2} /(1-\sqrt{\delta})).
        \end{equation*}
        That is, the solution converges to a critical point Q-linearly.
\end{proof}

\section{Proof of Lemma~\ref{lemma:GSspectral}: Divergence of Gauss-Seidel methods}\label{sec:GSspectral}

\begin{proof}
        Because the dynamics of the Gauss-Seidel method (GS) on the system 
        \begin{equation*}
                \min_{x\in\bR^n}\; f(x),\quad\text{where } f(x)\equiv x^T Sx,
        \end{equation*}
        has the same Jacobian as $J_{GS}$,
        proving $\rho(J_{GS})>1$ is equivalent to proving the 
        ``linear divergence'' of the Gauss-Seidel method, which cyclically optimizes each coordinate of $x\in\bR^n$.
        Further, since $S\not\succeq 0$, there is an eigenvector $q\in\bR^n$ of $S$ such that $q^T Sq <0$.

        Consider the sequence $\{x_r\}_{r=0,1,\ldots}$ generated by the GS.
        That is, $x_r=(J_{GS})^rx_0,\;\forall r>0$, where $(J_{GS})^r$ is $J_{GS}$ to the $r$-th power.
        Let the initial solution of the system be $x_0=q$ so that $f(x_0)<0$.
        Because the Gauss-Seidel method is greedy in every coordinate updates,
        it is monotonically deceasing in the function value.
        Thus, there are only two cases for the sequence of function values: 1) the function value converges below zero; 2) the function value goes to negative infinity.

        Denote $z^r_i$ the $x_r$ before updating the $i$-th
        coordinate and let $z^r_1 = x_r$ and $z^r_{n+1}=x_{r+1}$.
        This way, only the $i$-th coordinate between $z^r_i$ and $z^r_{i+1}$ is changed and the inner cyclic updates can be flattened as
        \begin{equation}\label{eq:GSspectral:3}
                x_r=z^r_1\rightarrow z^r_2\rightarrow\ldots\rightarrow z^r_n\rightarrow z^r_{n+1} = x_{r+1}.
        \end{equation}

        \paragraph{1) When the function value converges.} The monotonic decreasing property of GS
        implies that the function difference
        converges to zero. By the same analysis in
        Lemma~\ref{lemma-sufficient},\footnote{We can obtain the
                result by fixing $y$ in Lemma~\ref{lemma-sufficient} to be a constant and let $V\in \bR^{1\times n}$.
                The result can also be obtained by examining the coordinate
        updates of GS, which is already known in the literature.} we have
        \begin{equation*}
                f(x_{r}) - f(x_{r+1}) = \sum_{i=1}^ny_i\norm{z^r_i-z^r_{i+1}}^2.
        \end{equation*}
        Thus, the flattened sequence $\{z_i^r\}$ convergences, which implies $\{x^r\}$ also converges.
        Let $\bar{x}$ be the limit of the sequence $\{x^r\}$.
        Being a limit of the GS sequence means that $\bar{x}$ is a fixed point
        of GS,
        which implies $S\bar{x}=0$ and $f(\bar{X})=\bar{x}^T S\bar{x}=0$. 
        This contradicts with the monotonic decreasing property of GS and
        the fact that $f(x_0) < 0$.

        \paragraph{2) When the function value $x_r^T Sx_r$ goes to negative
        infinity.} Because the spectrum of $S$ is bounded, we
        know that $\norm{x_r}$ also goes to infinity.
        For simplicity, we focus on the $r$-th iterate and 
        write $z_i^r$ as $z_i$. From the GS, $z_{i,i}$ is updated to $z_{i+1,i}=\frac{-1}{y_i}\sum_j c_{ij}z_{j,i}$, and
        we have
        \begin{equation}\label{eq:zgrad}
                f(z_i)-f(z_{i+1})=y_i\norm{z_i-z_{i+1}}^2=y_i(z_{i,i}+\frac{1}{y_i}(\sum_j c_{ij} z_{j,i}))^2 = |e_i^T Sz_i|^2/y_i,
        \end{equation}
        where $e_i$ is the $i$-th coordinate vector and the first equality is from $f(x)=x^TSx$. 
        Then we have the following claim from
        \citet[cliam 1]{lee2017first}.
        \begin{claim}\label{claim}
                Assume $x_r$ be in the range of $S$. There exists an index $j$ such that $\frac{1}{y_j}|e_j^T S z_j|\geq \omega\norm{z_j}$ for some global constant $\omega>0$ that only depends on $S$ and $n$.

        \end{claim}
        The full proof of the claim is listed after this lemma for completeness.
        To fulfill the assumption, we can decompose
        $x_r=x_r^{\shortpar}+x_r^{\bot}$, where $x_r^{\shortpar}$ is in the
        range of $S$ and $x_r^{\bot}$ is in the null of $S$.
        Consider the flattened inner cyclic update $z^\shortpar_i$ 
        like \eqref{eq:GSspectral:3} but starting from $x_r^\shortpar$ such that\footnote{Note that only
        $x_r$ is decomposed to $x_r^\shortpar$ in $\textnormal{range}(S)$ and $x_r^\bot$ in $\textnormal{null}(S)$. Symbols $z_i^\shortpar$
        are the GS iterates generated from $x_r^\shortpar$ and might not be
        in the range of $S$.}
        \begin{equation*}
                x_r^\shortpar=z_1^\shortpar\rightarrow z_2^\shortpar \rightarrow \ldots 
                \rightarrow z_n^\shortpar\rightarrow z_{n+1}^\shortpar.
        \end{equation*}
        Because $J_{GS}$ map the null space of $S$ to
        itself,\footnote{Consider $p$ such that $Sp=0$.
        Then $(J_{GS})p = -(L+y)^{-1}L^T p = -(L+y)^{-1}Sp+p = p$.}
        \begin{equation*}
        f(x_r)-f(x_{r+1}) =
                f(x_r^\shortpar)-f(J_{GS}(x_r^\shortpar+x_r^\bot)) =
                f(z_1^\shortpar)-f(z_{n+1}^\shortpar+x_r^\bot) =
                f(z_1^\shortpar)-f(z_{n+1}^\shortpar).
        \end{equation*}
        Further, because GS is coordinate-wise monotonic decreasing and the function
        decrease of a coordinate update is smaller than the whole cyclic update, 
        by above equality and \eqref{eq:zgrad} we have
        \begin{equation*}
                f(x_r)-f(x_{r+1}) = f(z_1^\shortpar)-f(z_{n+1}^\shortpar)
                \geq \frac{(e_j^T S z_j^\shortpar)^2}{y_j}
                \geq y_j\omega^2 \norm{z_j^{\shortpar}}^2.
        \end{equation*}
        The last inequality is from Claim~\ref{claim}.
        Thus,
        \begin{equation}\label{eq:suffdec}
                f(x_{r+1}) \leq
                f(x_r) - y_j\omega^2\norm{z_j^\shortpar}^2 
        \end{equation}
        Further, because $\norm{z_j^\shortpar}^2 \geq |z_j^{\shortpar
        T}Sz_j^\shortpar|/\rho(S)$ and $z_j^{\shortpar T}Sz_j^\shortpar = f(z_j^\shortpar)\leq f(x_r^\shortpar) = f(x_r) \leq f(x_0) < 0$,
        \begin{equation}
                f(x_r) - y_j\omega^2\norm{z_j^\shortpar}^2 
                \leq f(x_r) + \frac{y_j\omega^2}{\rho(S)} z_j^{\shortpar
                T}Sz_j^\shortpar \leq
                (1+\frac{y_{min}\omega^2}{\rho(S)})f(x_r).
                \label{eq:expand}
        \end{equation}
        Combining \eqref{eq:suffdec} and \eqref{eq:expand}, we obtain the exponential divergence to negative infinity
        \begin{equation}\label{eq:GSspectral:1}
                f(x_{r+1}) \leq (1+\frac{y_{min}\omega^2}{\rho(S)})f(x_r) \leq
                (1+\frac{y_{min}\omega^2}{\rho(S)})^{r+1} f(x_0),\quad\forall r\geq 0.
        \end{equation}
        The last inequality is from applying the first inequality recursively.
        Because $S\succeq \sigma_{\min}(S)I_n$ and $\sigma_{min}(S)<0$,
        \begin{equation}\label{eq:GSspectral:2}
                f(x_r) = ((J_{GS})^rx_0)^TS((J_{GS})^rx_0) \geq \sigma_{min}(S) \norm{(J_{GS})^r x_0}^2 \geq \sigma_{min}(S)\norm{(J_{GS})^r}^2 \norm{x_0}^2.
        \end{equation}
        Combining \eqref{eq:GSspectral:1} and \eqref{eq:GSspectral:2}, we have
        \begin{equation*}
                \sigma_{\min}(S)\norm{(J_{GS})^r}^2\norm{x_0}^2 \leq \left(1+\frac{y_{\min}\omega^2}{\rho(S)}\right)^rf(x_0), \quad\forall r>0.
        \end{equation*}
        Applying Gelfand's theorem for spectral radius, we conclude that
        \begin{equation*}
                \rho(J_{GS}) = \lim_{r\appinf} \norm{(J_{GS})^{r}}^{1/r} \geq \sqrt{1+\frac{y_{min}\omega^2}{\rho(S)},}
        \end{equation*}
        which means that the spectral radius $\rho(J_{GS})$ is strictly larger
        than $1$.
        \end{proof}

\paragraph{Proof of the Claim~\ref{claim} in the above lemma.} \label{sec:GSspectral}
Note that the following proof is essentially the same with \citet[cliam 1]{lee2017first}, where their $\alpha$ is our $\frac{1}{y_i}$ and their $y_t$ is our $x_r$. 
The only difference here is that we prove the result for the exact Gauss-Seidel method, and they prove the result for the coordinate gradient descent method with a step size. The proof is listed here for completeness.
\begin{proof}
        We will prove by contradiction. Assume that 
        \begin{equation}\label{eq:assumeS}
                \frac{1}{y_j}|e_j^T S z_j|<\omega\norm{z_j}
        \text{ for all }j=1\ldots n \text{ for certain }\omega.
        \end{equation}
        Now we show the following result by induction, that for $j=2\ldots n+1$,
        \begin{equation}\label{eq:inddiff}
                \norm{x_r-z_{j}}<2(j-1)\omega\norm{x_r}.
        \end{equation}
        Remember from \eqref{eq:zgrad} we have
        \begin{equation}\label{eq:zgrad2}
                y_j\norm{z_j-z_{j+1}}^2 = |e_j^T Sz_j|^2/y_j.
        \end{equation}
        For $j=2$, we have the induction basis for \eqref{eq:inddiff} from the above equality and \eqref{eq:assumeS}, that
        \begin{equation*}
                \norm{x_r-z_2} = \norm{z_1-z_2} = \frac{1}{y_1}|e_1^T S z_1| < \omega \norm{z_1} = \omega \norm{x_r} < 2\omega\norm{x_r},
        \end{equation*}
        and accordingly $\norm{z_2} \leq \norm{z_2-z_1}+\norm{z_1} < (1+2\omega)\norm{x_r}$.
        Now we do the induction. Suppose the hypothesis \eqref{eq:inddiff} holds for a $j$. This implies 
        \begin{equation}\label{eq:indz}
                \norm{z_j}\leq\norm{z_j-x_r}+\norm{x_r}<(1+2(j-1)\omega)\norm{x_r}.
        \end{equation}
        Then at $j+1$,
        \begin{align*}
                \norm{x_r-z_{j+1}} &\leq \norm{x_r-z_j} +
                \norm{z_j-z_{j+1}}&&\text{by the triangular inequality}\\
                &< 2(j-1)\omega\norm{x_r} + \frac{1}{y_j} |e_j^T S z_j|&&\text{by hypothesis \eqref{eq:inddiff} at $j$ and \eqref{eq:zgrad2}}\\
                &< 2(j-1)\omega\norm{x_r} + \omega \norm{z_j}&&\text{by assumption \eqref{eq:assumeS}}\\
                &< 2(j-1)\omega\norm{x_r} +
                \omega(1+2(j-1)\omega)\norm{x_r}&&\text{by \eqref{eq:indz}}\\
                                  &\leq 2j\omega\norm{x_r},
        \end{align*}
        where the last inequality holds from picking $\omega\in(0,\frac{1}{2n})$ so that $\omega(1+2(j-1)\omega-2)<0$.
        Thus, the induction on \eqref{eq:inddiff} holds. With the result \eqref{eq:inddiff}, for $j=2\ldots n$ we have
        \begin{align}
                \frac{1}{y_{max}}|e_j^T Sx_r| 
                &\leq \frac{1}{y_j}|e_j^T Sx_r|&&\text{by $y_{max}\geq y_j$}\no\\
                &\leq \frac{1}{y_j} (|e_j^T Sz_j|+|e_j^TS(x_r-z_j)|)&&\text{by the triangular inequality}\no\\
                &< \omega \norm{z_j} + \frac{1}{y_j} \norm{Se_j}\norm{x_r-z_j}&&\text{by \eqref{eq:assumeS} and Cauchy inequality}\no\\
                &< \omega(1+2(j-1)\omega)\norm{x_r}+\frac{1}{y_j} 2(j-1)\omega\norm{Se_j} \norm{x_r}&&\text{by \eqref{eq:indz} and \eqref{eq:inddiff}}\no\\
                &\leq \omega(1+2n\omega+2n\frac{1}{y_{min}}\rho(S))\norm{x_r}\label{eq:xrBnd},
        \end{align}
        where the last inequality is from $\norm{Se_j}\leq \norm{S}\norm{e_j}\leq \rho(S)$,
        $y_{\min}\leq y_j$, and $j\leq n+1$.
        Note that the result of \eqref{eq:xrBnd} for $j=1$ also holds because \eqref{eq:assumeS} and $x_r=z_1$.
        Summing the square of \eqref{eq:xrBnd} over $j=1\ldots n$ and put it in a
        square root, we have 
        \begin{equation*}
        \sqrt{n}\omega(1+2n\omega+2n\frac{\rho(S)}{y_{min}})\norm{x_r}
        >
                \frac{1}{y_{max}}\norm{Sx_r} \geq \kappa_{\minnz}(S)\norm{x_r},
        \end{equation*}
        where $\kappa_{\minnz}(S)=\sqrt{\sigma_{\minnz}(S^TS)}>0$ is the minimum nonzero singular value of $S$ and the last inequality holds 
        because $\kappa_{\minnz}(S)\norm{x_r} \leq \norm{Sx_r}$ from $x_r\in\text{range}(S)$.
        Cancelling $\norm{x_r}$ from both sides of the above inequality, the left-hand side goes to $0$ when $\omega\rightarrow 0$ but the right-hand side stays constant.
        Thus, picking small enough $\omega$ such that $\kappa_{\minnz}(S) \geq y_{max}\sqrt{n}\omega (1+2n\omega+2n\frac{\rho(S)}{y_{min}})$
        leads to a contradiction.\footnote{Note that the choice of $\omega$ only depends on $n$ and $S$.} So the claim holds.
\end{proof}

\section{Proof of Theorem~\ref{thm:globalConvStepsize}: Global convergence with a step size}\label{sec:globalConvStepsize}
\begin{lemma} \label{lemma:MSnondegenerate}
        The Mixing method $M_\theta$ with a step size $\theta\in(0, \frac{1}{\max_i \norm{c_i}_1})$ never degenerates.
        That is, there is a constant $\delta\in(0,1)$ such that
        \begin{equation}
                \norm{\theta Vc_i}\leq 1-\delta<1\quad\text{and}\quad\norm{v_i-\theta Vc_i} \geq \delta>0.\no
        \end{equation}
\end{lemma}
\begin{proof}
        Taking a constant $\theta\in(0, \frac{1}{\max_i\norm{c_i}_1})$
	is equivalent to taking $\theta=\frac{1-\delta}{\max_i\norm{c_i}_1}$ for a constant $\delta\in(0,1)$.
        From the triangular inequality,	
        \begin{equation}
		 \norm{Vc_i}=\norm{\sum_j c_{ij}v_j}\leq \sum_j|c_{ij}|\norm{v_j}=\norm{c_i}_1.\no
        \end{equation}
        So we have $\norm{\theta Vc_i}\leq 1-\delta<1$. The second result follows from  $\norm{v_i-\theta Vc_i}\geq 1-\norm{\theta Vc_i}$.
\end{proof}
\begin{lemma}\label{lemma:diffeomorphism} The Mixing method $M_\theta$ with a step size $\theta\in(0, \frac{1}{\max_i\norm{c_i}_1})$ is a diffeomorphism.
\end{lemma}
\begin{proof}
        Note that $M_\theta$ can be decomposed to
        \begin{equation*}
                M_\theta(V) = \Phi^n(\Phi^{n-1}(\ldots\Phi^1(V))),
        \end{equation*}
        where the column update $\Phi^i(V):\bR^{k\times n}\rightarrow\bR^{k\times n}$ is defined as
        \begin{equation*}
                (\Phi^i(V))_{s=1\ldots n} = \begin{cases}
                        \frac{v_i - \theta Vc_i}{\norm{v_i - \theta Vc_i}}&\text{if }s=i\\
                        v_s&\text{otherwise}.
                \end{cases}
        \end{equation*}
        Thus, if we can prove that $\Phi^i(V)$ is a diffeomorphism for $i=1\ldots n$, 
        then $M_\theta$ is a diffeomorphism because compositions of diffeomorphisms are still diffeomorphism.
        Specifically, because all variables $v_j$ except for $v_i$ are given and stay the same,
        proving diffeomorphism of $\Phi^i(V)$ is equivalent to proving the diffeomorphism of the projective mapping $\phi:\bR^n\rightarrow\bR^n$
        \begin{equation*}
                \phi(v) = \frac{v-g}{\norm{v-g}},
        \end{equation*}
        where $g=\theta Vc_i$ is known. Let $\phi(v) = z$.
        We claim the inverse function $\phi^{-1}(z)$ is
        \begin{equation*}
        \phi^{-1}(z) = \alpha z+g,
        \quad\text{where }\;\alpha =-z^T g+\sqrt{(z^Tg)^2+1-\norm{g}^2}.
        \end{equation*}
        The square root is valid because of Lemma~\ref{lemma:MSnondegenerate}.
        We prove the claim by validation. First,
        \begin{equation*}
                \phi^{-1}(\phi(v)) = \phi^{-1}\left(\frac{v-g}{\norm{v-g}}\right).
        \end{equation*}
        By using the property that $\norm{v}=1$, the $\alpha$ for the above function is
        \begin{align*}
                \alpha &= \frac{-(v-g)^T g}{\norm{v-g}} + \sqrt{\left(\frac{(v-g)^T g}{\norm{v-g}}\right)^2+1-\norm{g}^2}\\
                       &= \frac{1}{\norm{v-g}}\left(-(v-g)^T g + \sqrt{((v-g)^T g)^2 + \norm{v-g}^2(1-\norm{g}^2)}\right)\\
                       &= \frac{1}{\norm{v-g}}\left(-v^T g+\norm{g}^2+1-v^T g\right)\\
                       &= \norm{v-g}.
        \end{align*}
        Thus, $\phi^{-1}(\phi(v)) = \norm{v-g} \frac{v-g}{\norm{v-g}}+g = v$
        is indeed the inverse function.
        The diffeomorphism follows from the smoothness of $\phi(v)$ and $\phi^{-1}(z)$ when $\norm{v-g}\geq \delta>0$ by Lemma~\ref{lemma:MSnondegenerate}.
\end{proof}
\vspace{-0.3cm}
\begin{lemma}\label{lemma:squaredDecreaseMS}For the Mixing method $M_\theta$ with step size $\theta\in(0,\frac{1}{\max_i\norm{c_i}})$,
        let $\hat{V}=M_\theta(V)$. Following the notation in \eqref{eq:g} and \eqref{eq:MS}, we have
        \begin{equation*}
                f(V)-f(\hat{V}) = \sum_i \frac{1+y_i}{\theta}\norm{v_i-\hat{v}_i}^2,\no
        \end{equation*}
        where $y_i=\norm{v_i-\theta g_i}$ and $g_i=\sum_{j<i}c_{ij}\hat{v}_j+\sum_{j>i}c_{ij}v_j$.
\end{lemma}
\begin{proof}
        Consider the inner iteration of the Mixing method with a step size.
        With the same analysis to Lemma~\ref{lemma-sufficient}, the function value before updating the variable $v_i$ is
        $2g_i^Tv_i$, and the function difference after updating $v_i$ to $\hat{v}_i=(v_i-\theta g_i)/y_i$ is
        \begin{align*}
                2g_i^T(v_i-\hat{v}_i) 
                &= 2(g_i+\frac{v_i-\theta g_i}{\theta})^T(v_i-\hat{v}_i) -2(\frac{v_i-\theta g_i}{\theta})^T(v_i-\hat{v}_i)\no\\
                &= 2\frac{1}{\theta}v_i^T(v_i-\hat{v}_i) - 2\frac{y_i}{\theta}\hat{v}_i^T(v_i-\hat{v_i})\no\\
                &= \frac{1+y_i}{\theta} 2(1-v_i^T\hat{v}_i)
                = \frac{1+y_i}{\theta}\norm{v_i-\hat{v_i}}^2.\no
        \end{align*}
        Thus, the result holds from summing the above equation over $i=1\ldots n$.
\end{proof}
\vspace{-0.3cm}
\paragraph{Proof of Theorem~\ref{thm:globalConvStepsize}} $ $\\
\vspace{-0.1cm}
\begin{proof}
        Similar to Lemma~\ref{lemma:Jacobian}, the Jacobian of the Mixing method $M_\theta$ with step size $\theta$ is
        \begin{equation*}
                J(V) = (D_y\otimes I_k-\theta P L\otimes I_k)^{-1} P \theta L^T\otimes I_k,\;\;\text{where }\;P=\diag(P_1,\ldots,P_n)\;\text{ and }\;P_i=I-\hat{v}_i\hat{v}_i^T.
        \end{equation*}
        By Lemma~\ref{lemma:diffproj}, the spectral radius of $J$ at a critical point is lower bounded by
        $J_{CGD}=(D_y-\theta L)^{-1} \theta L^T$, which is the Jacobian of the coordinate gradient descent method (CGD) on a linear system $S=C+D_y$. 
        Because the CGD admits a Jacobian with $\rho(J_{CGD})>1$ when $S\not\succeq 0$ \citep[Proposition 5]{lee2017first},
        it follows that all non-optimal critical points are unstable fixed points for $M_\theta$.
        Further, since the Mixing method $M_\theta$ with step size $\theta$ is a diffeomorphism by Lemma~\ref{lemma:diffeomorphism},
        we can apply the center-stable manifold theorem \citep[Theorem III.7]{shub2013global} to the mapping.
        To be specific, the corollary of center-stable manifold theorem in \citet[Theorem 2]{lee2017first}\footnote{Note
        that \citet[Lemma 1]{lee2017first} use only the property of diffeomorphism,
 	so their assumption on the non-singular Jacobian is not necessary. Actually, non-singular Jacobian is a sufficient condition 
	for the existence of one-to-one mapping but not the necessary condition.} implies that  
        the Mixing method $M_\theta$ escapes all non-optimal critical point almost surely under random initialization.\footnote{Note
	 that the critical points here is non-isolated because they are invariant to rotations in $R^k$.
	 While \citet[Theorem 2]{lee2017first} suffices for our result, interested reader can also refer to \cite{panageas2016gradient}
	 on how they use the Lindelöf lemma to solve the non-isolation issue.}
        Further, because $M_\theta$ is monotonically decreasing by Lemma~\ref{lemma:squaredDecreaseMS} and the objective value is lower bounded,
	$M_\theta$ converges to a first-order critical points (with the same analysis to Theorem~\ref{thm:critical mixing}).
        In conclusion, the almost surely divergence from the non-optimal critical points and the convergence to a critical point imply
        that the method converges to a global optimal solution almost surely.
\end{proof}

\section{Proof of Lemma~\ref{lemma:deficient}: Rank Deficiency in Critical Points}\label{sec:defficient}
The proof is a specialized version of \citep[Lemma 9]{boumal2016non} for the MAXCUT SDP, in which their $Y$ is our $V$ and their $\mu(V)$ is our $y$. We list it here for completeness.
\begin{proof}
        Let $V$ be a first-order critical point of problem \eqref{eq:unitsdp}, which means that there is a corresponding $y_i=\norm{Vc_i},\;i=1\ldots n$ such that
        \begin{equation*}
                VS = 0,\quad\text{where}\;S\equiv C+D_y.
        \end{equation*}
        This implies
        \begin{equation*}
                \textnormal{rank} (V) \leq \textnormal{null} (C+D_y) \leq \max_{\nu} \textnormal{null}(C+D_\nu).
        \end{equation*}
        Note that the right-hand side is independent of $V$, so we can use it to bound the rank of all critical $V$.
        Let $\nu$ be a solution of the right-hand side, $M\equiv C+D_\nu$, and $\textnormal{null}(M)\equiv\ell$.
        Writing $C=M-D_\nu$, we have
        \begin{equation*}
                C\in \mathcal{N}_\ell + im (D),
        \end{equation*}
        in which the $+$ denotes the set-sum, $im(D)$ denotes the image of all diagonal matrices of size $n$, and $\mathcal{N}_\ell$ denotes the set of symmetric matrices of size $n$ with nullity $\ell$.
        Because of the symmetricity of $\mathcal{N}_\ell$,
        \begin{equation*}
                \textnormal{dim}(\mathcal{N}_\ell) = \frac{n(n+1)}{2} - \frac{\ell(\ell+1)}{2}.
        \end{equation*}
        Further, because $\textnormal{rank}(V)\leq k$, we can assume that $\ell\geq k$. Union all possible $\ell$, 
        \begin{equation*}
                C\in \bigcup_{\ell=k\ldots n} \mathcal{N}_\ell + im(D).
        \end{equation*}
        Note that the right-hand side is now independent of $C$.
        Because the dimension of a finite union is at most the maximal dimension, and the dimension of a finite set sum is at most the sum of set dimensions,
        \begin{equation}\label{eq:defficient}
                \textnormal{dim}\left(\bigcup_{\ell\in k\ldots n} \mathcal{N}_\ell + im(D)\right) \leq \textnormal{dim}(\mathcal{N}_k + im(D)) \leq \frac{n(n+1)}{2} - \frac{k(k+1)}{2} + \textnormal{rank}(D).
        \end{equation}
        We know that $\textnormal{rank}(D)=n$ because the space of diagonal matrix has $n$ free dimensions.
        Because the symmetric matrix $C$ lives in the space $\frac{n(n+1)}{2}$, almost no $C$ satisfies the right-hand side of \eqref{eq:defficient} if we take large enough $k$ so that
        \begin{equation*}
                \frac{n(n+1)}{2} -\frac{k(k+1)}{2} + n < \frac{n(n+1)}{2}.
        \end{equation*}
        Thus, almost no $C$ has critical point of rank $k$ if $\frac{k(k+1)}{2}>n$, which means for almost all $C$, the critical point has at most rank $k-1$.
\end{proof}

\section{Application: Word embedding problem}\label{sec:wordemb}
The word embedding is a feature learning technique to embed the meanings of words
as low-dimensional vectors. 
For example, the popular Word2vec model
\citep{mikolov2013distributed,mikolov2013efficient} successfully embeds
similarities and analogies between words using a shallow neural network.
Another popular model, GloVe \citep{pennington2014glove}, uses
a factorization-based formulation and achieves better accuracy in
analogies tasks compared to Word2vec. 
The theoretical justifications of these two models are discussed in RANDWALK
\citep{arora2015rand}.  Here, we show that our coordinate descent approach
can also be applied to learning word embeddings with the GloVe objective function,
highlighting the fact that the Mixing methods can be applied to problems typically
considered in the domain of solely non-convex optimization.
\begin{algorithm}[t]
        \begin{algorithmic}[1]
        \State Initialize $v_i$ randomly on a unit sphere\;
        \State Initialize $b_i:=0$ for each $i=1,\ldots,m$\;
        \While{not yet converged}
                \For{$i=1,\ldots,n$}
                        \State Solve $Hd+g=0$ approximately by conjugate gradient method using
                        \State \quad$Hd := \sum_j w_{ij} (v_j^T d) v_j$ and $g := \sum_j w_{ij}e_{ij}v_j$, where
                        \State \quad$e_{ij} = v_i^T v_j + b_i + b_j - \log(c_{ij})$\;
                        \State Let $v_i := v_i + d$\;
                        \State Update bias term $b_i:= b_i - (\sum_{j}w_{ij}e_j)/(\sum_j w_{ij}^2)$\;
                \EndFor
        \EndWhile
        \end{algorithmic}
        \caption{The Mixing method for Word embedding problem}\label{alg:word-embedding}
\end{algorithm}

\paragraph{Problem description.} Let $C$ be the co-occurrence matrix such that $c_{ij}$ is the 
number of times word $i$ and $j$ co-occur within the same window in the corpus.
We consider solving a slightly modified version of the GloVe objective function \citep{pennington2014glove}
\begin{equation*} \min_{V \in \R^{k \times n}} \quad \frac{1}{2} \sum_{i\neq j}
w_{ij}\Big( v_i^Tv_j +b_i + b_j - \log c_{ij}\Big)^2,
\end{equation*} 
where $n$ is size of the vocabulary, $k$ is the number of latent factors
and $w_{ij} = \min\{C_{w,w'}^{3/4}, 100\}$ is a tuning factor to suppress 
high-frequency words. The only difference with GloVe is that we do not include the self-loop, i.e., $i=j$ terms, in the formulation.

\paragraph{Application of Mixing method.}
Focusing on the subproblem involving only variable $v_i$ and take a step $d\in\bR^k$, we can see that the subproblem $\min_d f(v_i+d)$ becomes:
\begin{equation*} \minimize_{d\in\bR^k}\; 
              \frac{1}{2} \sum_{j} w_{ij} \Big( (v_i +
d)^T(v_j) +b_i+b_j- \log c_{ij} \Big)^2
\end{equation*}
Define $e_{ij}=v_i^T v_j+b_i+b_j-\log c_{ij}$.
Then the above subproblem can be reformulated as
\begin{equation*}
        \minimize_{d\in\bR^k}\;\; \frac{1}{2}d^T\big(\overbrace{\sum_j w_{ij} v_j v_j^T}^H\big)d + \big(\overbrace{\sum_j w_{ij}e_{ij}v_j}^g\big)^T d,
\end{equation*}
which is an unconstrained quadratic program solvable in $O(n^3)$ time.
In practice, we apply
the conjugate gradient method with a stopping condition on 
$\norm{\gr_d f(v_i+d)}$ to obtain good enough approximations and
cyclically update all the $v_i$. Totally, updating all the
variables once takes $O(k\cdot m\cdot(\textnormal{\# CG iterations}))$ time,
where $m$ is the number of nonzeros in $C$ and typically $(\textnormal{\# CG iteration})<10$ in our
settings. Algorithm~\ref{alg:word-embedding} contains a complete description.

\begin{figure*}[t] \centering
        \begin{subfigure}[b]{0.4\textwidth}
        \centering
        \includegraphics[width=55mm,height=35mm]{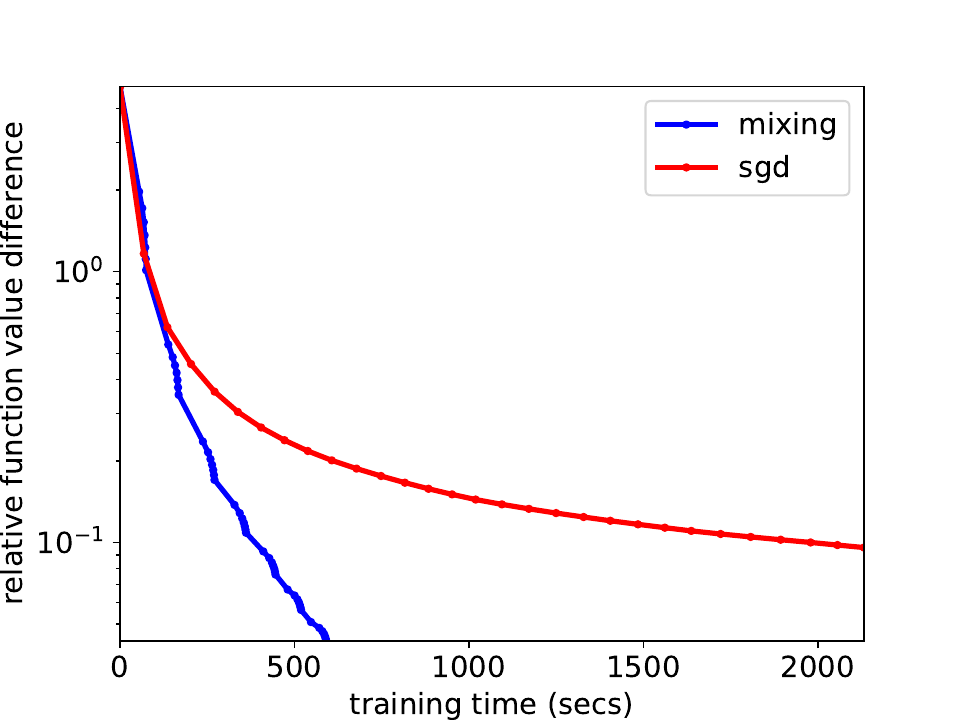}
        \caption{wiki8 (n=$1332$, nnz=$1678046$, k=$50$)}
        \end{subfigure}
        \hfill
        \begin{subfigure}[b]{0.45\textwidth}
        \centering
        \includegraphics[width=55mm,height=35mm]{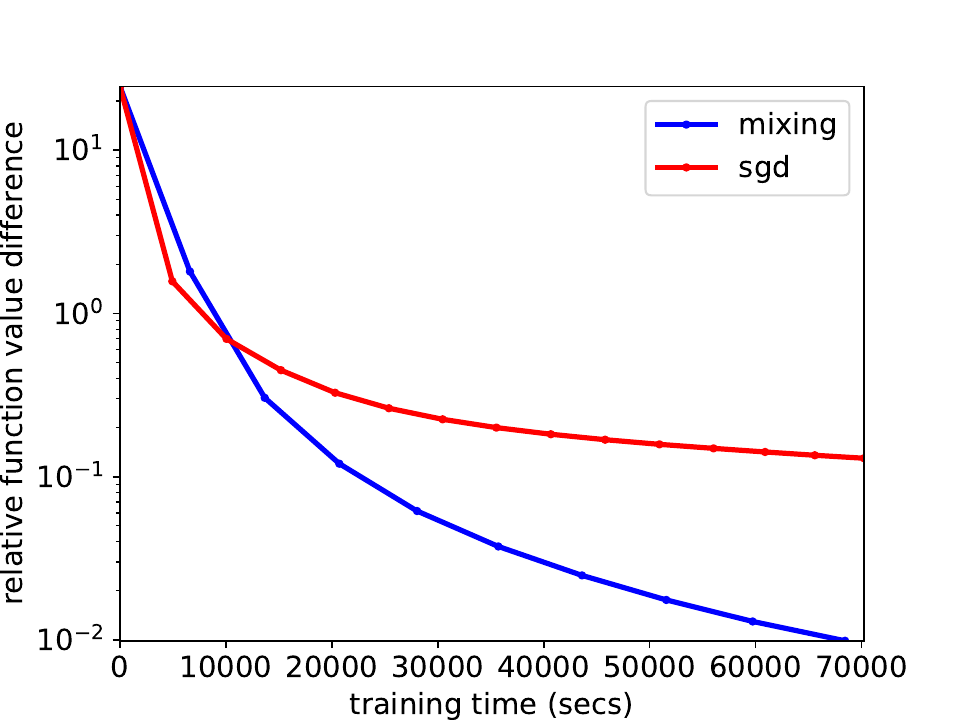}
        \caption{enwiki (n=$75317$, nnz=$875447516$, k=$300$)}
        \end{subfigure}
        \caption{$(f-f_{\min})$ v.s. training time for the word embedding problem, where $f_{\min}$ is minimum objective value we have.
        The experiments show that the Mixing method can be generalized to nonlinear objective function and is faster than SGD.}
        \label{fig:wordcost}
\end{figure*}

\paragraph{Results.}
Figure~\ref{fig:wordcost} shows the result of comparing the proposed Mixing
method with the stochastic gradient method, which is the default solver for GloVe.
We consider the wiki8 and the enwiki corpus, which are widely-used benchmarks for word
embeddings. The corpus is pre-processed following
\cite{pennington2014glove} (removing 
non-textual elements, sentence splitting, 
and tokenization), and words that appeared fewer than $1000$ times in the corpus
are ignored. Figure
\ref{fig:wordcost} shows the results of the Mixing method versus the SGD on the two
corpora.  For both datasets, the Mixing method converges substantially faster, achieving a lower
function value than the SGD method.

\end{document}